\documentclass[reqno]{amsart}
\usepackage{amssymb,amsmath,amsthm,amsfonts,color}

\usepackage{mathrsfs,dsfont, comment,mathscinet}
\usepackage{hyperref}
\usepackage{enumerate,esint}
\usepackage{mathtools}

\DeclarePairedDelimiter{\floor}{\lfloor}{\rfloor}
\allowdisplaybreaks
\numberwithin{equation}{section}
\theoremstyle{plain}
\newtheorem{theorem}{Theorem}[section]
\newtheorem{proposition}[theorem]{Proposition}
\newtheorem{lemma}[theorem]{Lemma}
\newtheorem{example}[theorem]{Example}
\newtheorem{corollary}[theorem]{Corollary}
\theoremstyle{definition}
\newtheorem{definition}[theorem]{Definition}

\newcommand\R{\mathbb R}
\newcommand\M{\mathbb M}
\newcommand\N{\mathbb N}
\newcommand\Oo{\mathcal{O}}
\newcommand\Z{\mathbb{Z}}

\newcommand\iO{\int_\Omega}
\newcommand{\iOQ}{\int_{\Omega}\int_Q}

\newcommand\iQ{\int_Q}
\newcommand\e{\varepsilon}

\newcommand\PP{\mathbb{P}}
\newcommand\QQ{\mathbb{Q}}
\newcommand\rd{\R^d}

\newcommand\rl{\R^l}

\newcommand\en{\e_n}

\newcommand\liminfn{\liminf_{n\to +\infty}}
\newcommand\ep{\varepsilon}
\newcommand{\scal}[2]{\langle #1,\,#2 \rangle}

\newcommand\wk{\rightharpoonup}
\newcommand\wkts{\overset{2-s}{\rightharpoonup}}
\newcommand\sts{\overset{2-s}{\to}}

\newcommand{\C}{\mathcal{C}}

\newcommand\Rn{\mathbb{R}^N}

\newcommand\pdeor{\mathscr{A}}

\newcommand{\A}{\mathbb{A}}

\newcommand{\cx}[0]{\mathcal{C}_{x}}
\newcommand{\qa}[0]{Q_{\pdeor}}
\newcommand{\iq}[0]{\int_Q}

\newcommand\be[1]{\begin{equation}\label{#1}}
\newcommand\ee{\end{equation}}
\newcommand\ba[1]{\begin{align}\label{#1}}
\newcommand\ea{\end{align}}
\newcommand\bas{\begin{align*}}
\newcommand\eas{\end{align*}}
\newcommand\nn{\nonumber}

\title [Homogenization for $\pdeor(x)-$quasiconvexity]{Homogenization for $\pdeor$-quasiconvexity with variable coefficients} 
\author[E. Davoli] {Elisa Davoli} 
\address[Elisa Davoli]{Department of Mathematics\\ University of Vienna\\Oskar-Morgenstern-Platz 1\\1090 Vienna (Austria)}
\email[E. Davoli]{elisa.davoli@univie.ac.at}
\author[I. Fonseca] {Irene Fonseca} 
\address[Irene Fonseca]{Department of Mathematics\\ Carnegie Mellon University\\Forbes Avenue\\Pittsburgh PA 15213}
\email[I. Fonseca]{fonseca@andrew.cmu.edu}

\begin{document} 
\vskip .2truecm
\begin{abstract}
\small{A homogenization result for a family of oscillating integral energies
$$u_{\ep}\mapsto\iO f(x,\tfrac{x}{\ep},u_{\ep}(x))\,dx,\quad \ep\to 0^+$$
is presented, where the fields $u_{\ep}$ are subjected to first order linear differential constraints depending on the space variable $x$. 
The work is based on the theory of $\pdeor$-quasiconvexity with variable coefficients and on two-scale convergence techniques, and generalizes the previously obtained results in the case in which the differential constraints are imposed by means of a linear first order differential operator with constant coefficients. The identification of the relaxed energy in the framework of $\pdeor$-quasiconvexity with variable coefficients is also recovered as a corollary of the homogenization result.
}
\end{abstract}
\maketitle
\section{Introduction}
In this paper we continue the study of the problem of finding an integral representation for limits of oscillating integral energies
$$u_{\ep}\mapsto \iO f\Big(x,\frac{x}{\ep^{\alpha}},u_{\ep}(x)\Big)\,dx,$$
where $\Omega\subset \R^N$ is a bounded open set, $\ep\to 0$, and the fields $u_{\ep}$ are subjected to $x-$dependent differential constraints of the type
\be{eq:constrain-no-divergence}
\sum_{i=1}^N  A^i\Big(\frac{x}{\ep^{\beta}}\Big)\frac{\partial u_{\ep}(x)}{\partial {x_i}}\to 0\quad\text{strongly in }W^{-1,p}(\Omega;\R^l),\,1<p<+\infty,
\ee
or in divergence form
\be{eq:constrain-divergence}
\sum_{i=1}^N \frac{\partial}{\partial {x_i}} \Big(A^i\Big(\frac{x}{\ep^{\beta}}\Big)u_{\ep}(x)\Big)\to 0\quad\text{strongly in }W^{-1,p}(\Omega;\R^l),\,1<p<+\infty,
\ee
with $A^i(x)\in Lin(\R^d;\R^l)$ for every $x\in\R^N$, $i=1,\cdots,N$, $d,l\geq 1$, and where $\alpha,\beta$ are two nonnegative parameters. Different regimes are expected to arise, depending on the relation between $\alpha$ and $\beta$. 

We recently analyzed in \cite{davoli.fonseca} the limit case in which $\alpha=0,\, \beta>0$, the energy density is independent of the first two variables, and the fields $\{u_{\ep}\}$ are subjected to \eqref{eq:constrain-divergence}. We will consider here the case in which $\alpha>0, \beta=0$ and \eqref{eq:constrain-no-divergence}, i.e., the energy density is oscillating but the differential constraint is fixed and in ``nondivergence" form. 
The situation in which there is an interplay between $\alpha$ and $\beta$ will be the subject of a forthcoming paper.\\

The key tool for our study is the notion of $\pdeor-$quasiconvexity with variable coefficients, characterized in \cite{santos}. $\pdeor-$quasiconvexity was first investigated by Dacorogna in \cite{dacorogna} and then studied by Fonseca and M\"uller in \cite{fonseca.muller} in the case of constant coefficients (see also \cite{fonseca.dacorogna}).  More recently, in \cite{santos}  Santos extended the analysis of \cite{fonseca.muller} to the case in which the coefficients of the differential operator $\pdeor$ depend on the space variable. 

In order to illustrate the main ideas of $\pdeor$-quasiconvexity, we need to introduce some notation.
For $i=1\cdots,N$, consider matrix-valued maps $A^i\in C^{\infty}(\R^N;\M^{l\times d})$, where for $l,d\in\N$, $\M^{l\times d}$ stands for the linear space of matrices with $l$ rows and $d$ columns, and for every $x\in \R^N$ define $\pdeor$ as the differential operator such that
\be{eq:intro-def-op}\pdeor u:=\sum_{i=1}^N A^i(x)\frac{\partial u(x)}{\partial {x_i}},\,x\in\Omega\ee
for $u\in L^1_{\rm loc}(\Omega; \R^d)$, where $\frac{\partial u}{\partial {x_i}}$ is to be interpreted in the sense of distributions.
We require that the operator $\pdeor$ satisfies a uniform constant-rank assumption (see \cite{murat}), i.e., there exists $r\in \N$ such that 
 \begin{equation}
 \label{cr}
 \text{rank }\sum_{i=1}^N A^i(x)w_i=r\quad\text{for every }w\in\mathbb{S}^{n-1},
 \end{equation} uniformly with respect to $x$, where $\mathbb{S}^{N-1}$ is the unit sphere in $\mathbb{R}^N$.

The definitions of $\pdeor$-quasiconvex function and $\pdeor$-quasiconvex envelope in the case of variable coefficients read as follows:
\begin{definition}
Let $f:\Omega\times\R^d \to \R$ be a  Carath\'eodory function, let $Q$ be the unit cube in $\R^N$  centered at the origin, 
$$Q=\Big(-\frac{1}{2},\frac{1}{2}\Big)^N,$$
and denote by  $C^{\infty}_{\rm per}(\R^N;\R^d)$ the set of smooth maps which are $Q$-periodic in $\R^N$. Consider the set
$$\cx:=\Big\{w\in C^{\infty}_{\rm per}(\R^N;\R^d):\,\int_Q{w(y)\,dy}=0,\quad \sum_{i=1}^N A^i(x)\frac{\partial w(y)}{\partial {y_i}}=0 \Big\}.$$
For a.e. $x\in\Omega$ and $\xi\in\R^d$, the \emph{$\pdeor-$quasiconvex envelope} of $f$ in $x\in\Omega$ is defined as
$$\qa f(x,\xi):=\inf\Big\{\iq f(x,\xi+w(y))\,dy:\,w\in\cx\Big\}.$$
$f$ is said to be \emph{$\pdeor$-quasiconvex} if $f(x,\xi)=\qa f(x,\xi)$ for a.e. $x\in\Omega$ and $\xi\in\R^d$.
\end{definition}
Denote by $\pdeor^c$ a generic differential operator, defined as in \eqref{eq:intro-def-op} and with constant coefficients, i.e. such that 
$$A^i(x)\equiv A^i_c\quad\text{for every }x\in\R^N,$$
with $A^i_c\in\M^{l\times d}$, $i=1,\cdots,N$.
We remark that when $\pdeor = \pdeor^c= \rm curl$, i.e., when $v=\nabla\phi$ for some $\phi\in	W^{ 1,1}_{\rm loc} (\Omega; \R^m )$, then $d=m\times N$, and $\pdeor$-quasiconvexity reduces to Morrey's notion of quasiconvexity (see \cite{acerbi.fusco, ball, marcellini, morrey}).

The first identification of the effective energy associated to periodic integrands evaluated along $\pdeor^c$-free fields was provided in \cite{braides.fonseca.leoni}, by Braides, Fonseca and Leoni. Their homogenization results were later generalized in \cite{fonseca.kromer}, where Fonseca and Kr\"omer worked under weaker assumptions on the energy density $f$.\\

This paper is devoted to extending the results in \cite{fonseca.kromer} to the framework of $\pdeor-$quasiconvexity with variable coefficients. To be precise, in \cite{fonseca.kromer} the authors studied the homogenized energy associated to a family of functionals of the type
$$F_{\ep}(u_{\ep}):=\int_{\Omega}f\big(x,\tfrac{x}{\ep},u_{\ep}(x)\big)\,dx,$$
where $\Omega$ is a bounded, open subset of $\R^N$, $u_{\ep}\wk u$ weakly in $L^p(\Omega;\R^d)$  and the sequence $\{u_{\ep}\}$ satisfies a differential constraint of the form $\pdeor^c u_{\ep}=0$ for every $\ep$. 

We analyze the analogous problem in the case in which $\pdeor$ depends on the space variable and the differential constraint is replaced by the condition
$$\pdeor u_{\ep}\to 0\quad\text{strongly in }W^{-1,p}(\Omega;\R^l).$$ 
Our analysis leads to a limit homogenized energy of the form:
$$\mathscr{E}_{\rm hom}(u):=\begin{cases}\int_{\Omega} f_{\rm hom}(x,u(x))\,dx&\text{if }\pdeor u=0,\\
+\infty&\text{otherwise in }L^p(\Omega;\R^d),\end{cases}$$
where $\mathscr{W}$ is the class of maps $w\in L^p(\Omega;L^p_{\rm per}(\R^N;\R^d))$ such that 
$$\int_Q w(x,y)\,dy=0\quad\text{for a.e. }x\in\Omega,$$ and
\be{eq:condition-w}
\sum_{i=1}^N A^i(x)\frac{\partial w(x,y)}{\partial {y_i}}=0\quad\text{in }W^{-1,p}(Q;\R^l)\,\text{ for a.e. }x\in\Omega,\ee
and $f_{\rm hom}:\Omega\times \R^d\to [0,+\infty)$ is defined as
$$f_{\rm hom}(x,\xi):=\liminfn\inf_{v\in\C_x}\int_Q f(x,ny,\xi+v(y))\,dy.$$
Our main result is the following.
\begin{theorem}
\label{thm:main}
Let $1<p<+\infty$. Let $A^i\in C^{\infty}_{\rm per}(\R^N;\M^{l\times d})$, $i=1,\cdots, N$, and assume that $\pdeor$ satisfies the constant rank condition \eqref{cr}. Let $f:\Omega\times \Rn\times \rd$ be a function satisfying 
\ba{eq:hp-f-1}
 & f(x,\cdot,\xi)\quad\text{ is measurable},\\
 &\label{eq:hp-f-2} f(\cdot,y,\cdot)\quad\text{ is continuous},\\
 &\label{eq:hp-f-3} f(x,\cdot,\xi)\quad\text{ is }Q-\text{periodic},\\
 &\label{eq:growth-p-f-3}0\leq f(x,y,\xi)\leq C(1+|\xi|^p)\quad\text{for all }(x,\xi)\in\Omega\times\rd,\text{ and for a.e. }y\in\Rn.
 \end{align}
Then for every $u\in L^p(\Omega;\rd)$ there holds
\begin{multline*}
\inf\Big\{\liminf_{\ep\to 0}\iO f\Big(x,\frac{x}{\ep},u_{\ep}(x)\Big)\,dx:u_{\ep}\wk u\quad\text{weakly in }L^p(\Omega;\rd)\\
\text{and }\pdeor u_{\ep}\to 0\quad\text{strongly in }W^{-1,p}(\Omega;\rl)\Big\}\\
=\inf\Big\{\limsup_{\ep\to 0}\iO f\Big(x,\frac{x}{\ep},u_{\ep}(x)\Big)\,dx:u_{\ep}\wk u\quad\text{weakly in }L^p(\Omega;\rd)\\
\text{and }\pdeor u_{\ep}\to 0\quad\text{strongly in }W^{-1,p}(\Omega;\rl)\Big\}=\mathscr{E}_{\rm hom}(u).
\end{multline*}
\end{theorem}
As in \cite{davoli.fonseca} and \cite{fonseca.kromer}, the proof of this result is based on the \emph{unfolding operator}, introduced in \cite{cioranescu.damlamian.griso08,cioranescu.damlamian.griso} (see also \cite{visintin, visintin1}). In contrast with \cite[Theorem 1.1]{fonseca.kromer} (i.e. the case in which $\pdeor=\pdeor^c$), here we are unable to work with exact solutions of the system $\pdeor u_{\ep}=0$, but instead we consider sequences of asymptotically $\pdeor-$vanishing fields. This is due to the fact that for $\pdeor$-quasiconvexity with variable coefficients we do not project directly on the kernel of the differential constraint, but construct an ``approximate" projection operator $P$ such that for every field $v\in L^p$, the $W^{-1,p}$ norm of $\pdeor Pv$ is controlled by the $W^{-1,p}$ norm of $v$ itself (for a detailed explanation we refer to \cite[Subsection 2.1]{santos}). 

In \cite{davoli.fonseca} the issue of defining a projection operator was tackled by imposing an additional invertibility assumption on $\pdeor$ and by exploiting the divergence form of the differential constraint. We do not add this invertibility requirement here, instead we use the fact that
 in our framework the differential operator depends on the ``macro" variable $x$ but acts on the ``micro" variable $y$ (see \eqref{eq:condition-w}). Hence it is possible to define a pointwise projection operator $\Pi(x)$ along the argument of \cite[Lemma 2.14]{fonseca.kromer} (see Lemma \ref{lemma:proj-operator}). 

As a corollary of our main result we recover an alternative proof of the relaxation theorem \cite[Theorem 1.1]{braides.fonseca.leoni} in the framework of $\pdeor-$quasiconvexity with variable coefficients, 
that is we obtain the identification (see Corollary \ref{thm:relax})
$$\int_D \qa f(x,u(x))\,dx=\mathcal{I}(u,D)$$
for every open subset $D$ of $\Omega$, and for every $u\in L^p(\Omega;\rd)$ satisfying $\pdeor u=0$, where the functional $\mathcal{I}$ is defined as
\ba{eq:def-i-intro}
&\mathcal{I}(u,D):=\inf\Big\{\liminf_{\ep\to 0}\int_{D}f(x,u_{\ep}(x)):\, \\
\nn&\quad u_{\ep}\wk u\quad\text{weakly in }L^p(\Omega;\R^m)\,\text{and }\pdeor u_{\ep}\to 0\quad\text{strongly in }W^{-1,p}(\Omega;\R^l)\Big\}.
\end{align}
We point out here that a proof of this relaxation theorem follows directly combining \cite[Proof of Theorem 1.1]{braides.fonseca.leoni} with the arguments in \cite{santos}.
The interest in Corollary \ref{thm:relax} lies in the fact that it is obtained as a by-product of our homogenization result, and thus by adopting a completely different proof strategy.\\

In analogy to \cite{davoli.fonseca} one might expect to be able to apply an approximation argument and extend the results in Theorem \ref{thm:main-result-A-free} to the situation in which $A^i\in W^{1,\infty}(\R^N;\M^{l\times d})$, $i=1\cdots,N$, which is the least regularity assumption in order for $\pdeor$ to be well defined as a differential operator from $L^p$ to $W^{-1,p}$. We were unable to achieve this generalization, mainly because the projection operator here plays a key role in the proof of both the liminf and the limsup inequalities. In order to work with approximant operators $\pdeor^k$ having smooth coefficients, we would need to finally project on the kernel of $\pdeor$, whereas the projection argument provided in \cite{santos} applies only to the case of smooth differential constraints.\\

The article is organized as follows. In Section \ref{section:prel} we establish the main assumptions on the differential operator $\pdeor$ and we recall some preliminary results on two-scale convergence. In Section \ref{section:pt-quas} we recall the definition of $\pdeor$-quasiconvex envelope and we construct some examples of $\pdeor-$quasiconvex functions. Section \ref{section:A-free} is devoted to the proof of our main result.\\

\noindent\textbf{Notation}\\
Throughout this paper, $\Omega\subset \R^N$ is a bounded open set, $\mathcal{O}(\Omega)$ is the set of open subsets of $\Omega$,  $Q$ denotes the unit cube in $\R^N$ centered at the origin and with normals to its faces parallel to the vectors in the standard orthonormal basis of $\R^N$, 
$\{e_1,\cdots,e_N\}$, i.e.,
$$Q=\Big(-\frac{1}{2},\frac{1}{2}\Big)^N.$$
Given $1<p<+\infty$, we denote by $p'$ its conjugate exponent, that is 
$$\frac{1}{p}+\frac{1}{p'}=1.$$
Whenever a map $u\in L^p, C^{\infty},\cdots$, is $Q-$periodic, that is
$$u(x+e_i)=u(x)\quad i=1,\cdots, N$$
for a.e. $x\in \R^N$, we write $u\in L^p_{\rm per}, C^{\infty}_{\rm per},\cdots$, respectively. We will implicitly identify the spaces $L^p(Q)$ and $L^p_{\rm per}(\R^N)$. We will designate by $\scal{\cdot}{\cdot}$ the duality product between $W^{-1,p}$ and $W^{1,p'}_0$.

We adopt the convention that $C$ will denote a generic constant, whose value may change from expression to expression in the same formula.

\section{Preliminary results}
\label{section:prel}
In this section we introduce the main assumptions on the differential operator $\pdeor$ and we recall some preliminary results about $\pdeor-$quasiconvexity and two-scale convergence.
\subsection{Preliminaries}
\label{subsection:prel}
For $i=1,\cdots, N$, consider the matrix-valued functions $A^i\in C^{\infty}(\R^N;\M^{l\times d})$. For $1<p<+\infty$ and $u\in L^p(\Omega;\R^d)$, we set
$$\pdeor u:=\sum_{i=1}^N A^i(x)\frac{\partial u(x)}{\partial x_i} \in W^{-1,p}(\Omega;\R^l).$$
For every $x_0\in\Omega$ and $u\in L^p(\Omega;\rd)$ we define
$$\pdeor(x_0) u:=\sum_{i=1}^N A^i(x_0)\frac{\partial u(x)}{\partial x_i} \in W^{-1,p}(\Omega;\R^l).$$
We will also consider the operators
$$\pdeor_x w:=\sum_{i=1}^N A^i(x)\frac{\partial w(x,y)}{\partial {x_i}}$$
and
$$\pdeor_y w:=\sum_{i=1}^N A^i(x)\frac{\partial w(x,y)}{\partial {y_i}}$$
for every $w\in L^p(\Omega\times Q;\R^d)$.
Finally, for every $x_0\in\Omega$ and for $w\in L^p(\Omega\times Q;\R^d)$, we set
$$\pdeor_x (x_0) w:=\sum_{i=1}^N A^i(x_0)\frac{\partial w(x,y)}{\partial {x_i}}$$
and
$$\pdeor_y (x_0)w:=\sum_{i=1}^N A^i(x_0)\frac{\partial w(x,y)}{\partial {y_i}}.$$

For every $x\in \R^N$, $\lambda\in \R^N\setminus \{0\}$, let $\A(x,\lambda)$ be the linear operator
$$\A(x,\lambda):=\sum_{i=1}^N A^i(x)\lambda_i\in \M^{l\times d}.$$
We assume that $\pdeor$ satisfies the following \emph{constant rank condition}:
\be{eq:constant-rank-condition}
\text{rank }\Big(\sum_{i=1}^N A^i(x)\lambda_i\Big)=r\quad\text{for some }r\in N\text{ and for all }x\in\R^N,\lambda\in \Rn \setminus \{0\}.
\ee
For every $x\in\Rn$, $\lambda\in \Rn\setminus\{0\}$, let $\PP(x,\lambda):\rd\to\rd$ be the linear projection on Ker $\A(x,\lambda)$, and let $\QQ(x,\lambda):\rl\to\rd$ be the linear operator given by
\begin{eqnarray*}
&&\QQ(x,\lambda)\A(x,\lambda)\xi:=\xi-\PP(x,\lambda)\xi\quad\text{for all }\xi\in\rd,\\
&&\QQ(x,\lambda)\xi=0\quad\text{if }\xi\notin \text{Range }\A(x,\lambda).
\end{eqnarray*}
The main properties of $\PP(\cdot,\cdot)$ and $\QQ(\cdot,\cdot)$ are stated in the following proposition (see \cite[Subsection 2.1]{santos}).
\begin{proposition}
\label{prop:properties-P-Q}
Under the constant rank condition \eqref{eq:constant-rank-condition}, for every $x\in\Rn$ the operators $\PP(x,\cdot)$ and $\QQ(x,\cdot)$ are, respectively, $0-$homogeneous and $(-1)-$homogeneous. Moreover,
$\PP\in C^{\infty}(\Rn\times \Rn\setminus\{0\};\M^{d\times d})$ and $\QQ\in C^{\infty}(\Rn\times\Rn\setminus\{0\};\M^{d\times l})$.
\end{proposition}
\subsection{Two-scale convergence}
We recall here the definition  and some properties of two-scale convergence. For a detailed treatment of the topic we refer to, e.g., \cite{allaire, lukkassen.nguetseng.wall, nguetseng}. Throughout this subsection $1<p<+\infty$.
\begin{definition}
If $v\in L^p(\Omega\times Q;\rd)$ and $\{u_{\ep}\}\in L^p(\Omega;\rd)$, we say that $\{u_{\ep}\}$ \emph{weakly two-scale converge to} $v$ in $L^p(\Omega\times Q;\rd)$, $u_{\ep}\wkts v$, if
$$\iO u_{\ep}(x)\cdot \varphi\Big(x,\frac{x}{\ep}\Big)\,dx\to \iOQ v(x,y)\cdot \varphi(x,y)\,dy\,dx$$
for every $\varphi\in L^{p'}(\Omega;C^{\infty}_{\rm per}(\Rn;\rd))$.\\

We say that $\{u_{\ep}\}$ \emph{strongly two-scale converge to $v$} in $L^p(\Omega\times Q;\rd)$, $u_{\ep}\sts v$, if $u_{\ep}\wkts v$ and $$\lim_{\ep\to 0}\|u_{\ep}\|_{L^p(\Omega;\rd)}=\|v\|_{L^p(\Omega\times Q;\rd)}.$$
\end{definition}
Bounded sequences in $L^p(\Omega;\rd)$ are pre-compact with respect to weak two-scale convergence. To be precise (see \cite[Theorem 1.2]{allaire}), 
\begin{proposition}
\label{prop:2-scale-compactness}
Let $\{u_{\ep}\}\subset L^p(\Omega;\rd)$ be bounded. Then, there exists $v\in L^p(\Omega\times Q;\rd)$ such that, up to the extraction of a (non relabeled) subsequence, $u_{\ep}\wkts v$ weakly two-scale in $L^p(\Omega\times Q;\R^d)$, and, in particular
$$u_{\ep}\wk \iQ v(x,y)\,dy\quad\text{weakly in }L^p(\Omega;\rd).$$
\end{proposition}
The following result will play a key role in the proof of the limsup inequality (see \cite[Proposition 2.4, Lemma 2.5 and Remark 2.6]{fonseca.kromer}).
\begin{proposition}
\label{prop:simple-2-scale}
Let $v\in L^p(\Omega;C_{\rm per}(\Rn;\rd))$ or $v\in L^p_{\rm per}(\Rn;C(\overline{\Omega};\rd))$. Then, the sequence $\{v_{\ep}\}$ defined as
$$v_{\ep}(x):=v\Big(x,\frac{x}{\ep}\Big)$$
is $p-$equiintegrable, and
$$v_{\ep}\sts v\quad\text{strongly two-scale in }L^p(\Omega;\rd).$$ 
\end{proposition}
\subsection{The unfolding operator}
\label{subsection:unfolding}
We collect here the definition and some properties of the \emph{unfolding operator} (see e.g. \cite{cioranescu.damlamian.griso, cioranescu.damlamian.griso08, visintin, visintin1}).
\begin{definition}
Let $u\in L^p(\Omega;\rd)$. For every $\ep>0$, the unfolding operator
 $T_{\e}:L^p(\Omega;\rd)\to L^p(\R^N;L^p_{\rm per}(\Rn;\rd))$ is defined componentwise as
\be{eq:unfolding-operator}
T_{\e}(u)(x,y):=u\Big(\e\floor[\Big]{\frac{x}{\e}}+\e(y-\floor{y})\Big)\quad\text{for a.e. }x\in\Omega\text{ and }y\in\Rn,\ee
where $u$ is extended by zero outside $\Omega$ and $\floor{\cdot}$ denotes the least integer part. 
\end{definition}
The next proposition and the subsequent theorem allow to express the notion of two-scale convergence in terms of $L^p$ convergence of the unfolding operator.
\begin{proposition}
\label{prop:isometry} (see \cite{cioranescu.damlamian.griso, visintin1})
$T_{\ep}$ is a nonsurjective linear isometry from $L^p(\Omega;\rd)$ to $L^p(\R^N\times Q;\rd)$.
\end{proposition}
The following theorem provides an equivalent characterization of two-scale convergence in our framework (see \cite[Proposition 2.5 and Proposition 2.7]{visintin1},\cite[Theorem 10]{lukkassen.nguetseng.wall}).
\begin{theorem}
\label{thm:equivalent-two-scale}
Let $\Omega$ be an open bounded domain and let $v\in L^p(\Omega\times Q;\R^d)$. Assume that $v$ is extended to be $0$ outside $\Omega$. Then the following conditions are equivalent:
\begin{enumerate} 
\item[(i)]$u_{\ep}\wkts v\quad\text{weakly two scale in }L^p(\Omega\times Q;\rd),$
\item[(ii)] $T_{\ep}u_{\ep}\wk v$ weakly in $L^p(\R^N\times Q;\rd)$.
\end{enumerate}
Moreover, $$u^{\ep}\sts v\quad\text{strongly two scale in }L^p(\Omega\times Q;\rd)$$
if and only if
$$T_{\ep}u_{\ep}\to v\quad\text{strongly in }L^p(\R^N\times Q;\rd).$$
\end{theorem}
The following proposition is proved in \cite[Proposition A.1]{fonseca.kromer}.
\begin{proposition}
\label{prop:conv-unf-op}
For every $u\in L^p(\Omega;\R^d)$ (extended by $0$ outside $\Omega$),
$$\|u-T_{\ep}u\|_{L^p(\R^N\times Q;\R^d)}\to 0$$
as $\ep\to 0$.
\end{proposition}

\section{$\pdeor$-quasiconvex functions}
\label{section:pt-quas}
In this section we recall the notion of $\pdeor$-quasiconvexity and $\pdeor$-quasiconvex envelope, and we provide some examples of $\pdeor$-quasiconvex functions in the case in which $\pdeor$ has variable coefficients.

We start by recalling the main definitions when $\pdeor=\pdeor^c$, where $\pdeor^c$ is a first order differential operator with constant coefficients, that is, for every $u\in L^p(\Omega;\R^d)$,
$$\pdeor^c u(x):=\sum_{i=1}^N A^i_c\frac{\partial u(x)}{\partial x_i}\in W^{-1,p}(\Omega;\R^l),$$
with $A^i_c\in \M^{l\times d}$ for $i=1,\cdots, N$.

\begin{definition}
Let $f:\Omega\times \rd\to [0,+\infty)$ be a Carath\'eodory function, let $\pdeor^c$ be a first order differential operator with constant coefficients, and consider the set
$$\mathcal{C}_{\rm const}:=\Big\{w\in C^{\infty}_{\rm per}(\R^N;\R^d):\,\int_Q{w(y)\,dy}=0\quad\text{and}\quad \sum_{i=1}^N A^i_c\frac{\partial w(y)}{\partial y_i}=0\Big\}.$$
The \emph{$\pdeor^c$-quasiconvex envelope} of $f$ is the function $Q^{\pdeor^c}f:\Omega\times\R^d\to [0,+\infty)$, given by
\be{eq:def-A-quasiconvex-envelope-const}
Q^{\pdeor^c}f(x,\xi):=\inf\Big\{\iq f(x,\xi+w(y))\,dy:\quad w\in\mathcal{C}_{\rm const}\Big\},
\ee
for a.e. $x\in\Omega$ and for all $\xi\in \R^d$.

We say that $f$ is \emph{$\pdeor^c$-quasiconvex} if
$$f(x,\xi)=Q^{\pdeor^c} f(x,\xi)\quad\text{for a.e. }x\in\Omega\text{ and for all } \xi\in\rd.$$
\end{definition}

Similarly, in the case in which $\pdeor$ depends on the space variable, the definitions of $\pdeor$-quasiconvex envelope and $\pdeor$-quasiconvex function read as follows. 

\begin{definition}
Let $f:\Omega\times \rd\to [0,+\infty)$ be a Carath\'eodory map, let $\pdeor$ be a first order differential operator with variable coefficients, and for every $x\in\Omega$ consider the set
\be{eq:def-c-x}
\cx:=\Big\{w\in C^{\infty}_{\rm per}(\R^N;\R^d):\,\int_Q{w(y)\,dy}=0\quad\text{and}\quad \sum_{i=1}^N A^i(x)\frac{\partial w(y)}{\partial y_i}=0 \Big\}.
\ee The \emph{$\pdeor$-quasiconvex envelope} of $f$ is the function $\qa f:\Omega\times\R^d\to[0,+\infty)$, given by
\be{eq:def-A-quasiconvex-envelope}
\qa f(x,\xi):=\inf\Big\{\iq f(x,\xi+w(y))\,dy:\quad w\in \cx\Big\},
\ee
for a.e. $x\in\Omega$ and for all $\xi\in \R^d$.

We say that $f$ is \emph{$\pdeor$-quasiconvex} if
$$f(x,\xi)=\qa f(x,\xi)\quad\text{for a.e. }x\in\Omega\text{ and for all } \xi\in\rd.$$
\end{definition}
We finally introduce, for every $x_0\in\Omega$, the notion of pointwise $\pdeor(x_0)$-quasiconvex envelope and pointwise $\pdeor(x_0)$-quasiconvexity.
\begin{definition}
Let $x_0\in\Omega$. Let $f:\Omega\times \rd\to [0,+\infty)$ be a Carath\'eodory map and let $\pdeor$ be a first order differential operator with variable coefficients. The \emph{pointwise $\pdeor(x_0)$-quasiconvex envelope} of $f$ is the function $Q^{\pdeor(x_0)}f:\Omega\times\R^d\to[0,+\infty)$ given by
$$Q^{\pdeor(x_0)}f(x,\xi):=\inf\Big\{\iq f(x,\xi+w(y))\,dy:\quad w\in \mathcal{C}_{x_0}\Big\}$$
for a.e. $x\in\Omega$ and for all $\xi\in\rd$.

We say that $f$ is pointwise $\pdeor(x_0)$-quasiconvex in $x_0\in\Omega$, if
$$f(x,\xi)=Q^{\pdeor(x_0)}f(x,\xi)\quad\text{for a.e. }x\in\Omega\text{ and for all }\xi\in\rd.$$
\end{definition}
We stress that $\pdeor$-quasiconvexity and pointwise $\pdeor(x)$-quasiconvexity are related by the following ``fixed point" relation:
\be{eq:fixed-point-id}\qa f(x,\xi)=Q^{\pdeor(x)}f(x,\xi)\quad\text{for a.e. }x\in\Omega\text{ and for all }\xi\in\R^d.\ee

The remaining part of this section is devoted to illustrating these concepts with some explicit examples of $\pdeor$-quasiconvex functions. We first exhibit an example where $\pdeor$-quasiconvexity reduces to ${\pdeor^c}$-quasiconvexity for a suitable operator ${\pdeor^c}$ with constant coefficients.
\begin{example}
\label{ex:basic}
Let $1<p<+\infty$ and define
$$f(x,\xi):=a(x)b(\xi)\quad\text{for a.e }x\in\Omega\text{ and every }\xi\in\R^d,$$ with $a\in L^p(\Omega)$ and $b\in C(\R^d)\cap L^p(\R^d)$. In order for $f$ to be $\pdeor$-quasiconvex, the function $b$ must satisfy
$$b(\xi)=\inf\Big\{\iq b(\xi+w(y))\,dy:\,\quad w\in \cup_{x\in\Omega}\,\cx\Big\}.$$
Consider the case in which $\cx$ is the same for every $x$, for example when the differential constraint is provided by the operator:
$$\pdeor w(x):=\sum_{i=1}^NM(x)A^i_c\frac{\partial w(x)}{\partial x_i},$$
where $M:\Omega\to \M^{l\times l}$ and $\det M(x)>0$ for every $x\in\Omega$. In this case, 
$$\cx=\Big\{w\in C^{\infty}_{\rm per}(\Rn;\rd):\iq w(y)\,dy=0\text{ and }{\pdeor^c}w=0\Big\}$$  for every $x$, where 
$$\pdeor^c w(x):=\sum_{i=1}^NA^i_c\frac{\partial w(x)}{\partial x_i}.$$ Hence, $f$ is $\pdeor$-quasiconvex if and only if $b$ is $\pdeor^c$-quasiconvex.
\end{example}
In the previous example, $\pdeor$-quasiconvexity could be reduced to $\pdeor^c$-quasiconvexity owing to the fact that the class $\cx$ was constant in $x$. We provide now an example where an analogous phenomenon occurs, despite the fact that $\cx$ varies with respect to $x$.
To be precise, we consider the case in which $\pdeor$ is a smooth perturbation of the divergence operator. In this situation, the $\pdeor$-quasiconvex envelope of $f$ coincides with its convex envelope.
\begin{example}
\label{ex:traditional}
We consider a smooth perturbation of the divergence operator in a set $\Omega\subset \R^2$, that is
$$\pdeor u:=\Big(\begin{array}{cc}a(x)& 0\\0&1\end{array}\Big)\nabla u\quad\text{for every }u\in L^p(\Omega;\R^2),\quad 1<p<+\infty,$$
with $a\in C(\overline{\Omega})$  and 
$$\tfrac 12\leq a(x)<1\quad\text{for every }x\in\Omega.$$
We notice that
$$\ker \A(x,\lambda)=\Big\{\xi\in\R^2:a(x)\lambda_1\xi_1+\lambda_2\xi_2=0\Big\},$$
and therefore 
$${\rm rank }\,\A(x,\lambda)=1\quad\text{for every }x\in\Omega\text{ and }\lambda\in \R^2\setminus\{0\}.$$
In this situation, the class $\cx$ depends on $x$, since we have
$$\cx=\Big\{w\in C^{\infty}_{\rm per}(\R^2;\R^2):\quad\int_Q w(y)\,dy=0\quad\text{and}\quad a(x)\frac{\partial w_1(y)}{\partial y_1}+\frac{\partial w_2(y)}{\partial y_2}=0\Big\},$$
 although $$\cup_{\lambda\in\mathbb{S}^1}\ker \A(x,\lambda)=\R^2.$$
Let now $f:\rd\to [0,+\infty)$ be a continuous map.
By \cite[Proposition 3.4]{fonseca.muller}, it follows that
$\qa f(\xi)$ coincides with the convex envelope of $f$ evaluated at $\xi$, exactly as in the case of the divergence operator (see \cite[Remark 3.5 (iv)]{fonseca.muller}).
\end{example}
We conclude this section with an example in which the notion of $\pdeor$-quasiconvexity can not be reduced to $\pdeor^c$-quasiconvexity with respect to a constant operator $\pdeor^c$. 
\begin{example}
\label{ex:nontraditional}
Here we consider a smooth perturbation of the {\rm curl} operator in a set $\Omega\subset \R^2$. Let $\pdeor :L^p(\Omega;\R^2)\to W^{-1,p}(\Omega;\R^4)$ be given by
$$\pdeor u=\sum_{i=1,2}A^i(x)\frac{\partial u}{\partial x_i}\quad\text{for }u\in L^p(\Omega;\R^2),\quad1<p<+\infty,$$ 
where
$$A^i_{(j,k),(q)}(x):=a_i(x)\delta_{ij}\delta_{qk}-a_i(x)\delta_{ik}\delta_{qj}\,\text{ for every }x\in\Omega,\quad1\leq j,k,q\leq 2,$$

with $a_2(x)\equiv 1$, and $a_1\in C(\overline{\Omega})$ satisfies
$$\tfrac 12\leq a_1(x)\leq 1\quad\text{for every }x\in{\Omega}.$$
We first notice that 
\begin{eqnarray*}
\ker\mathbb{A}(x,\lambda)&:=&\Big\{\xi\in \R^2:\, \lambda_1 a_1(x)\xi_2=\lambda_2 \xi_1\Big\}\\
&=&\Big\{\xi\in \R^2:\,\xi_2=\alpha \lambda_2\text{ and } \xi_1=\alpha a_1(x)\lambda_1,\,\alpha\in\R\Big\},
\end{eqnarray*}
hence 
$${\rm rank }\,\mathbb{A}(x,\lambda)=3\quad\text{for every }x\in\Omega,\,\lambda\in\mathbb{S}^1.$$
The class $\cx$ depends on $x$ and there holds
\ba{eq:cx-example}
&\cx=\Big\{w\in C^{\infty}_{\rm per }(\R^2;\R^2):\,\int_Q w(y)\,dy=0\quad\text{and}\quad a_1(x)\frac{\partial w_{2}(y)}{\partial y_1}=\frac{\partial w_{1}(y)}{\partial y_2}\Big\}\\
\nn&\quad=\Big\{w\in C^{\infty}_{\rm per }(\R^2;\R^2):\,\int_Q w(y)\,dy=0,\,w_1(y)=a_1(x)\frac{\partial \varphi(y)}{\partial y_1},\\
\nn&\qquad\text{and}\quad w_2(y)=\frac{\partial \varphi(y)}{\partial y_2},\,\text{where }\varphi\in C^{\infty}_{\rm per}(\R^2)\Big\}.
\end{align}

Let now $g:\Omega\times \R^2\to [0,+\infty)$ be a quasiconvex function and let \mbox{$f:\Omega\times \R^2\to [0,+\infty)$} be defined as
$$f(x,\xi):=g\Big(x,\Big(\frac{\xi_1}{a_1(x)},\xi_2\Big)\Big) \text{ for a.e. }x\in\Omega\text{ and for every }\xi\in\R^2.$$
We claim that
$$\qa f(x,\xi)=f(x,\xi)\quad\text{for a.e. }x\in\Omega\text{ and for every }\xi\in\R^2.$$
Indeed, by \eqref{eq:cx-example} there holds
\bas
&\inf_{w \in \cx} \int_Q f(x,\xi+w(y))\,dy\\
&\quad=\inf\Big\{\int_Q f\Big(x,\Big(\xi_1+a_1(x)\frac{\partial \varphi_1(y)}{\partial y_1},\xi_2+\frac{\partial \varphi_2(y)}{\partial y_2}\Big)\Big)\,dy:\,\varphi\in C^{\infty}_{\rm per}(\R^2)\Big\}\\
&\quad=\inf\Big\{\int_Q g\Big(x,\Big(\frac{\xi_1}{a_1(x)}+\frac{\partial\varphi(y)}{\partial y_1},\xi_2+\frac{\partial\varphi(y)}{\partial y_2}\Big)\Big)\,dy:\,\varphi\in C^{\infty}_{\rm per}(\R^2)\Big\}\\
&\quad=Qg\Big(x,\Big(\frac{\xi_1}{a_1(x)},\xi_2\Big)\Big)
\end{align*}
for a.e. $x\in\Omega$ and for every $\xi\in\R^2$,
where $Qg$ denotes the quasiconvex envelope of the function $g$. The claim follows by the definition of $f$ and the quasiconvexity of $g$.
\end{example}

 \section{A homogenization result for $\pdeor$-free fields}
 \label{section:A-free}
In this section we prove a homogenization result for oscillating integral energies under weak $L^p$ convergence of $\pdeor-$vanishing maps. Fix $1<p<+\infty$ and consider a function $f:\Omega\times\Rn\times\rd\to[0,+\infty)$ satisfying \eqref{eq:hp-f-1}--\eqref{eq:growth-p-f-3}.
 
 
 In analogy with the case of constant coefficients (see \cite[Definition 2.9]{fonseca.kromer}), we define the class of $\pdeor$-free fields as the set
 \be{eq:def-A-f}\mathscr{F}:=\Big\{v\in L^p(\Omega\times Q;\R^d):\,\pdeor_y v=0\quad\text{and }\pdeor_x \iQ v(x,y)\,dy=0\Big\},\ee
 where both the previous differential conditions are in the sense of $W^{-1,p}$.
 
 We aim at obtaining a characterization of the homogenized energy
 \ba{eq:liminf-to-charact}
 &\inf\Big\{\liminf_{\ep\to 0} \iO f\Big(x,\frac{x}{\ep},u_{\ep}(x)\Big)\,dx:\,\,u_{\ep}\wk u\quad\text{weakly in }L^p(\Omega;\R^d)\\
 \nn&\quad\text{and }\pdeor u_{\ep}\to 0\quad\text{strongly in }W^{-1,p}(\Omega;\rl)\Big\}.
 \end{align}
 
 We start with a preliminary lemma, which will allow us to define a pointwise projection operator. We will be using the notation introduced in Sections \ref{section:prel} and \ref{section:pt-quas}. 
 
 \begin{lemma}
 \label{lemma:proj-operator}
 Let $1<p<+\infty$. Let $A^i\in C^{\infty}_{\rm per}(\R^N;\M^{l\times d})$, $i=1\cdots, N$, and assume that the associated first order differential operator $\pdeor$ satisfies \eqref{eq:constant-rank-condition}. Then, for every $x\in\Omega$ there exists a projection operator 
 $$\Pi(x): L^p(Q;\rd)\to L^p(Q;\rd)$$ such that 
 \begin{enumerate}
\item[(P1)]$\Pi(x)$ is linear and bounded, and vanishes on constant maps,
\item[(P2)]$\Pi(x)\circ \Pi(x)\psi(y)=\Pi(x)\psi(y)\quad\text{and}\quad\pdeor_y(x) (\Pi(x)\psi(y))=0\quad\text{in }W^{-1,p}(Q;\rl)$ for a.e. $x\in\Omega$, {for every }$\psi\in L^p(Q;\rd)$,
\item[(P3)]there exists a constant $C=C(p)>0$, independent of $x$, such that 
$$\|\psi(y)-\Pi(x)\psi(y)\|_{L^p(Q;\rd)}\leq C\|\pdeor_y(x) \psi(y)\|_{W^{-1,p}(Q;\rl)}$$
{for a.e. }$x\in\Omega$, { for every }$\psi\in L^p(Q;\rd)$ with $\iQ\psi(y)\,dy=0$,
\item[(P4)]if $\{\psi_n\}$ is a bounded $p$-equiintegrable sequence in $L^p(Q;\R^d)$, then $\{\Pi(x)\psi_n(y)\}$ is a $p$-equiintegrable sequence in $L^p(\Omega\times Q;\R^d)$,
\item[(P5)] if $\varphi\in C^1(\Omega;C^{\infty}_{\rm per}(\Rn;\rd))$ then the map $\varphi_{\Pi}$, defined by
 $$\varphi_{\Pi}(x,y):=\Pi(x)\varphi(x,y)\quad\text{for every }x\in\Omega\text{ and }y\in\R^N,$$
 satisfies $\varphi_{\Pi}\in C^1(\Omega; C^{\infty}_{\rm per}(\R^N;\rd))$.
 \end{enumerate}
   \end{lemma}
 \begin{proof}
 For every $x\in\Omega$, let $\Pi(x)$ be the projection operator provided by \cite[Lemma 2.14]{fonseca.muller}. 
 Properties (P1) and (P2) follow from \cite[Lemma 2.14]{fonseca.muller}. 
 
 In order to prove (P3), fix $x\in\Omega$ and $\psi\in C^{\infty}_{\rm per}(\Rn;\rd)$. Let $\A$, $\PP$ and $\QQ$ be the operators defined in Subsection \ref{subsection:prel}. Writing the operator $\Pi(x)$ explicitly, we have
 $$\Pi(x)\psi(y):=\sum_{\lambda\in\mathbb{Z}^N\setminus\{0\}}\PP(x,\lambda)\hat{\psi}(\lambda)e^{2\pi i y\cdot\lambda},$$
 where
 $$\hat{\psi}(\lambda):=\iq \psi(y)e^{-2\pi i y\cdot\lambda}\,dy,\quad\text{for every }\lambda\in\Z^N\setminus\{0\}$$
 are the Fourier coefficients associated to $\psi$. 
 
  By the (-1)-homogeneity of the operator $\mathbb{Q}$ (see Proposition \ref{prop:properties-P-Q}) we deduce 
  \ba{eq:proj-ast1}
  &|\psi(y)-\Pi(x)\psi(y)|^p=\Big|\sum_{\lambda\in\Z^N\setminus\{0\}}\mathbb{Q}(x,\lambda)\mathbb{A}(x,\lambda)\hat{\psi}(\lambda)e^{2\pi iy\cdot\lambda}\Big|^p\\
  \nn&\quad=\Big|\sum_{\lambda\in\Z^N\setminus\{0\}}\frac{1}{|\lambda|}\mathbb{Q}\Big(x,\frac{\lambda}{|\lambda|}\Big)\mathbb{A}(x,\lambda)\hat{\psi}(\lambda)e^{2\pi iy\cdot\lambda}\Big|^p.
  \end{align}
  For $1<p<2$, by the smoothness of $\mathbb{Q}$ and by applying first H\"older's inequality and then Hausdorff-Young inequality, we obtain the estimate
  \ba{eq:proj-ast2}
 & |\psi(y)-\Pi(x)\psi(y)|^p\\
 \nn&\quad\leq \Big(\max_{\lambda\in\Z^N,\,|\lambda|=1}\|Q(x,\lambda)\|\Big)^p\Big(\sum_{\lambda\in\Z^N\setminus\{0\}}\frac{1}{|\lambda|^p}\Big)\Big(\sum_{\lambda\in\Z^N\setminus\{0\}}|\mathbb{A}(x,\lambda)\hat{\psi}(\lambda)|^{p'}\Big)^{\frac{p}{p'}}\\
  &\nn\quad\leq C\Big(\sum_{\lambda\in\Z^N\setminus\{0\}}|\mathbb{A}(x,\lambda)\hat{\psi}(\lambda)|^{p'}\Big)^{\frac{p}{p'}}\leq C\|\pdeor_y(x)\psi(y)\|^p_{L^p(Q;\R^d)},
  \end{align}
  where we used the fact that 
  \be{eq:proj-ast3}
  |\mathbb{A}(x,\lambda)\hat{\psi}(\lambda)|\leq C|\widehat{\pdeor_y(x)\psi(y)}(\lambda)|
  \ee
  for every $x\in\Omega$ and $\lambda\in\Z^N\setminus\{0\}$, by the definition of the Fourier coefficients, and where both constants in \eqref{eq:proj-ast2} and \eqref{eq:proj-ast3} are independent of $\lambda$ and $x$.
  
  Consider now the case in which $p\geq 2$. By \eqref{eq:proj-ast1} we have
  \ba{eq:proj-star1}
  & |\psi(y)-\Pi(x)\psi(y)|^p\\
 \nn&\quad\leq \Big(\max_{\lambda\in\Z^N,\,|\lambda|=1}\|Q(x,\lambda)\|\Big)^p\Big(\sum_{\lambda\in\Z^N\setminus\{0\}}\frac{1}{|\lambda|^{p'}}\Big)^{\frac{p}{p'}}\Big(\sum_{\lambda\in\Z^N\setminus\{0\}}|\mathbb{A}(x,\lambda)\hat{\psi}(\lambda)|^{p}\Big)\\
  &\nn\quad\leq C\Big(\sum_{\lambda\in\Z^N\setminus\{0\}}|\mathbb{A}(x,\lambda)\hat{\psi}(\lambda)|^{p}\Big)\\
  &\nn\quad\leq C\Big(\sup_{\lambda\in\Z^N\setminus\{0\}}|\mathbb{A}(x,\lambda)\hat{\psi}(\lambda)|^{p-2}\Big)\sum_{\lambda\in\Z^N\setminus\{0\}}|\mathbb{A}(x,\lambda)\hat{\psi}(\lambda)|^{2}.
 \end{align}
By the definition of Fourier coefficients and by H\"older's inequality we have
$$|\mathbb{A}(x,\lambda)\hat{\psi}(\lambda)|\leq C\|\pdeor_y(x)\psi(y)\|_{L^p(Q;\R^l)}$$  
for every $x\in\Omega$ and $\lambda\in\Z^N\setminus\{0\}$. In view of \cite[Theorem 2.9]{fonseca.muller},
$$\sum_{\lambda\in\Z^N\setminus\{0\}}|\mathbb{A}(x,\lambda)\hat{\psi}(\lambda)|^{2}=\|\pdeor_y(x)\psi(y)\|^2_{L^2(Q;\R^l)}.$$
Therefore by \eqref{eq:proj-star1}, applying again H\"older's inequality,
\ba{eq:proj-point1}
& |\psi(y)-\Pi(x)\psi(y)|^p\\
 &\nn\quad\leq C\|\pdeor_y(x)\psi(y)\|_{L^p(Q;\R^l)}^{p-2}\|\pdeor_y(x)\psi(y)\|^2_{L^2(Q;\R^l)}\leq C\|\pdeor_y(x)\psi(y)\|_{L^p(Q;\R^l)}^{p}
\end{align}
where the constant $C$ is independent of $x$ and $y$. Property (P3) follows by \eqref{eq:proj-ast2} and \eqref{eq:proj-point1} via a density argument.
 
 (P4) follows directly from (P3), arguing as in the proof of \cite[Lemma 2.14 (iv)]{fonseca.muller}.
 
 Let now $\varphi\in C^1(\Omega;C^{\infty}_{\rm per}(\Rn;\rd))$. The regularity of the map $\varphi_{\Pi}$ is a direct consequence of Proposition \ref{prop:properties-P-Q}, the definition of $\Pi$ and the regularity of $\pdeor$. Indeed,
 \be{eq:direct-def}\varphi_{\Pi}(x,y):=\sum_{\lambda\in\Z^N\setminus\{0\}}\PP(x,\lambda)\hat{\varphi}(x,\lambda)e^{2\pi i y\cdot\lambda},\ee
 for every $x\in\Omega$ and $y\in \R^N$, where
$$\hat{\varphi}(x,\lambda):=\iq \varphi(x,\xi)e^{-2\pi i\xi\cdot\lambda}\,d\xi$$
 for every $x\in\Omega$ and $\lambda\in\Z^N\setminus\{0\}$. By the regularity of $\varphi$ and by \cite[Theorem 2.9]{fonseca.muller} we obtain the estimate
 $$\Big(4\pi^2\sum_{\lambda\in\Z^N\setminus\{0\}}|\lambda|^2|\hat{\varphi}(x,\lambda)|^2\Big)^{\tfrac12}\leq\Big\|\sum_{i=1}^N\frac{\partial\varphi}{\partial y_i}(x,y)\Big\|_{L^2(Q;\R^d)}\leq C$$
 for every $x\in\Omega$, hence by Proposition \ref{prop:properties-P-Q} and Cauchy-Schwartz inequality there holds
 \ba{eq:proj-point2}
 &\sum_{\lambda\in\Z^N\setminus\{0\},\,|\lambda|\geq n}|\mathbb{P}(x,\lambda)\hat{\varphi}(x,\lambda)e^{2\pi i y\cdot\lambda}|\leq C\sum_{\lambda\in\Z^N\setminus\{0\},\,|\lambda|\geq n}|\hat{\varphi}(x,\lambda)|\\
 &\nn\quad\leq C\Big(\sum_{\lambda\in\Z^N\setminus\{0\},\,|\lambda|\geq n}|\hat{\varphi}(x,\lambda)|^2|\lambda|^2\Big)^{\tfrac12}\Big(\sum_{\lambda\in\Z^N\setminus\{0\},\,|\lambda|\geq n}\frac{1}{|\lambda|^2}\Big)^{\tfrac12}\\
&\nn\quad \leq C\Big(\sum_{\lambda\in\Z^N\setminus\{0\},\,|\lambda|\geq n}\frac{1}{|\lambda|^2}\Big)^{\tfrac12}.
 \end{align}
 By \eqref{eq:proj-point2} the series in \eqref{eq:direct-def} is uniformly convergent, and hence $\varphi_{\Pi}$ is continuous. The differentiability of $\varphi_{\Pi}$ follows from an analogous argument.
   \end{proof}
For every $v\in L^p(\Omega\times Q;\rd)$, let
\ba{eq:def-s-v}
\mathcal{S}_v:=&\Big\{\{u_{\ep}\}\subset L^p(\Omega;\rd): u_{\ep}\wkts v\quad\text{weakly two-scale in }L^p(\Omega\times Q;\R^d)\\
\nn&\quad\text{and }\pdeor u_{\ep}\to 0\quad\text{strongly in }W^{-1,p}(\Omega;\rl)
\Big\}.
\end{align}
Let also
\be{eq:def-s}
\mathcal{S}:=\Big\{\cup\mathcal{S}_v: v\in L^p(\Omega\times Q;\rd)\Big\}.\ee
  We provide a characterization of the set $\mathcal{S}$. \begin{lemma}
 \label{lemma:A-free-fields}
 Let $v\in L^p(\Omega\times Q;\R^d)$. Let $\pdeor$ be a first order differential operator with variable coefficients, satisfying \eqref{eq:constant-rank-condition}. The following conditions are equivalent:
 \begin{enumerate}
 \item[(C1)] $v\in \mathscr{F}$ (see \eqref{eq:def-A-f});\\
 \item[(C2)] $\mathcal{S}_v$ is nonempty.
 \end{enumerate}
 \end{lemma}
 \begin{proof}
 We first show that (C2) implies (C1). Let $v\in L^p(\Omega\times Q;\rd)$ and let $\{u_{\ep}\}\subset \mathcal{S}_v$. Consider a test function $\psi\in W^{1,p}_0(\Omega;\rl)$. Then
 $$\scal{\pdeor u_{\ep}}{\psi}\to 0\quad\text{as }\ep\to 0.$$
 On the other hand,
 \be{eq:free-field-eq}
 \scal{\pdeor u_{\ep}}{\psi}:=-\sum_{i=1}^N\iO\Big(\frac{\partial A^i(x)}{\partial x_i}u_{\ep}(x)\cdot\psi(x)+A^i(x)u_{\ep}(x)\cdot\frac{\partial \psi(x)}{\partial x_i}\Big)\,dx
 \ee
 and by Proposition \ref{prop:2-scale-compactness}, up to the extraction of a (not relabeled) subsequence,
 \be{eq:free-field-wk}
 u_{\ep}\wk \iQ v(x,y)\,dy\quad\text{weakly in }L^p(\Omega;\rd).\ee
Passing to the limit in \eqref{eq:free-field-eq} yields
 $$\pdeor_x \iQ v(x,y)\,dy=0\quad\text{in }W^{-1,p}(\Omega;\rl).$$
 
 In order to deduce the second condition in the definition of $\mathscr{F}$, we consider a sequence of test functions $\big\{\ep\varphi\big(\frac{x}{\ep}\big)\psi(x)\big\}$, where $\varphi\in W^{1,p'}_0(Q;\rd)$ and $\psi\in C^{\infty}_c(\Omega)$. Since this sequence is uniformly bounded in $W^{1,p'}_0(\Omega;\rd)$, we have
 $$\scal{\pdeor u_{\ep}}{\ep\varphi\Big(\frac{\cdot}{\ep}\Big)\psi}\to 0$$
 as $\ep\to 0$, where
\bas
 &\scal{\pdeor u_{\ep}}{\ep\varphi\Big(\frac{\cdot}{\ep}\Big)\psi}\\
 &\quad=-\ep\sum_{i=1}^N\iO\Big(\frac{\partial A^i(x)}{\partial x_i}u_{\ep}(x)\cdot \varphi\Big(\frac{x}{\ep}\Big)\psi(x)+A^i(x)u_{\ep}(x)\cdot\varphi\Big(\frac{x}{\ep}\Big)\frac{\partial\psi(x)}{\partial x_i}\Big)\,dx\\
&\qquad - \sum_{i=1}^N\iO A^i(x)u_{\ep}(x)\cdot\frac{\partial\varphi}{\partial x_i}\Big(\frac{x}{\ep}\Big)\psi(x)\,dx.
 \end{align*}
 Passing to the subsequence of $\{u_{\ep}\}$ extracted in \eqref{eq:free-field-wk}, the first line of the previous expression converges to zero. By the definition of two-scale convergence, the second line converges to
 $$- \sum_{i=1}^N\iOQ A^i(x)v(x,y)\cdot\frac{\partial \varphi(y)}{\partial y_i}\psi(x)\,dy\,dx,$$
and thus
 $$\scal{\sum_{i=1}^N A^i(x)\frac{\partial v(x,\cdot)}{\partial y_i}}{\varphi}_{W^{-1,p}(Q;\rl),W^{1,p'}_0(Q;\rl)}=0$$
 for a.e. $x\in\Omega$, that is
 $$\pdeor_y v=0\quad\text{in }W^{-1,p}(Q;\rl)\quad\text{for a.e. }x\in\Omega.$$
 This completes the proof of (C1).\\
 
Assume now that (C1) holds true, i.e.,  $v\in \mathscr{F}$. In order to construct the sequence $\{u_{\ep}\}$, set
 $$v_1(x,y)=v(x,y)-\iQ v(x,z)\,dz.$$
 We first assume that $v_1\in C^1(\Omega;C^1_{\rm per}(\Rn;\rd))$. Defining
 $$u_{\ep}(x):=\iQ v(x,y)\,dy+v_1\Big(x,\frac{x}{\ep}\Big)\quad\text{for a.e. }x\in\Omega,$$ 
 by Proposition \ref{prop:simple-2-scale} we have
 $u_{\ep}\sts v$ strongly two-scale in $L^p(\Omega\times Q;\R^d)$. Moreover, by the definition of $\mathscr{F}$ and Propositions \ref{prop:2-scale-compactness} and \ref{prop:simple-2-scale},
\bas
 \sum_{i=1}^N A^i(x)\frac{\partial u_{\ep}(x)}{\partial x_i}&=\sum_{i=1}^N A^i(x)\frac{\partial}{\partial x_i}\Big(\iQ v(x,y)\,dy+v_1\Big(x,\frac{x}{\ep}\Big)\Big)\\
 &=\sum_{i=1}^N A^i(x)\frac{\partial v}{\partial x_i}\Big(x,\frac{x}{\ep}\Big)\wk \sum_{i=1}^N A^i(x)\frac{\partial}{\partial {x_i}}\iQ v(x,y)\,dy=0
 \end{align*}
 weakly in $L^p(\Omega;\rl)$.
 Hence $v$ satisfies $(C2)$. 
 
In the general case in which $v_1\in L^p(\Omega\times Q;\rd)$, we first need to approximate $v_1$ in order to keep the periodicity condition during the subsequent regularization. To this purpose, we extend $v_1$ to $0$ outside $\Omega\times Q$, we consider a sequence $\{\varphi_j\}\in C^{\infty}_c(Q)$ such that $0\leq \varphi_j\leq 1$ and $\varphi_j\to 1$ pointwise, and we define the maps
 $$v_1^j(x,y):=\varphi_j(y)v_1(x,y)\quad\text{for a.e. }x\in\Omega\quad\text{and }y\in Q.$$
Extend these maps to $\Omega\times \R^N$ by periodicity. It is straightforward to see that
 \be{eq:double-lp-conv}
 v_1^j\to v_1\quad\text{strongly in }L^p(\Omega;L^p_{\rm per}(\R^N;\rd))
 \ee
 by the dominated convergence theorem. Moreover
 \be{eq:double-right-conv}
 \|\pdeor_y v_1^j\|_{W^{-1,p}(Q;\rl)}\to 0\quad\text{strongly in }L^p(\Omega).
 \ee
 Indeed, by \eqref{eq:double-lp-conv}, and since $\pdeor_y v=0$,
 $$\|\pdeor_y(x) v_1^j(x,\cdot)\|_{W^{-1,p}(Q;\rl)}\to 0\quad\text{for a.e. }x\in\Omega,$$
and
 $$\|\pdeor_y(x) v_1^j(x,\cdot)\|^p_{W^{-1,p}(Q;\rl)}\leq C\|\varphi_j\|^p_{L^{\infty}(Q)}\|v_1(x,y)\|^p_{L^p(Q;\rd)}\leq C\|v_1(x,y)\|^p_{L^p(Q;\rd)}$$
 for a.e. $x\in\Omega$. Thus \eqref{eq:double-right-conv} follows by the dominated convergence theorem.\\
 
  Convolving first with respect to $y$ and then with respect to $x$ we construct a sequence $\{v_1^{\delta,j}\}\in C^{\infty}(\Omega; C^{\infty}_{\rm per}(\Rn;\rd))$ such that
 \be{eq:v-uno-delta-wk}
 v_1^{\delta,j}\to v_1^j\quad\text{strongly in }L^p(\Omega;L^p_{\rm per}(\Rn;\rd)),
 \ee
 and
 \be{eq:approx-double-right}
 \|\pdeor_y v_1^{\delta,j}\|_{W^{-1,p}(Q;\rl)}\to 0\quad\text{strongly in }L^{p}(\Omega),\ee
 as $\delta\to 0$.
In view of \eqref{eq:double-lp-conv}--\eqref{eq:approx-double-right}, a diagonal argument provides a subsequence $\{\delta(j)\}$ such that $\{v_1^{\delta(j),j}\}$ satisfies 
\be{eq:limit-set1}
v_1^{\delta(j),j}\to v_1\quad\text{strongly in }L^p(\Omega;\,L^p_{\rm per}(\R^N;\R^d))
\ee
and
\be{eq:limit-set2}
\|\pdeor_yv_1^{\delta(j),j}\|_{W^{-1,p}(Q;\R^l)}\to 0\quad\text{strongly in }L^p(\Omega).
\ee
Set
$$w^j(x,y):=\Pi(x)\Big(v_1^{\delta(j),j}(x,y)-\iQ v_1^{\delta(j),j}(x,y)\,dy\Big)$$
for a.e. $x\in\Omega$ and $y\in Q$. By Lemma \ref{lemma:proj-operator} we have $w^j\in C^{\infty}(\Omega;\,C^{\infty}_{\rm per}(\R^N;\R^d))$,
\be{eq:limit-set1p}
\pdeor_y w^j=0\quad\text{in }W^{-1,p}(Q;\R^l)\quad\text{for a.e. }x\in\Omega,
\ee
and
\bas
&\|w^j-v_1\|^p_{L^p(\Omega\times Q;\R^d)}\leq C\Big(\Big\|w^j-v_1^{\delta(j),j}+\iQ v_1^{\delta(j),j}(x,y)\,dy\Big\|^p_{L^p(\Omega\times Q;\R^d)}\\
&\nn\qquad+\|v_1^{\delta(j),j}-v_1-\iQ v_1^{\delta(j),j}(x,y)\,dy\|^p_{L^p(\Omega\times Q;\R^d)}\Big)\\
&\nn\quad\leq C\Big(\int_{\Omega}\|\pdeor_y(x) v_1^{\delta(j),j}(x,y)\|^p_{W^{-1,p}(Q;\R^l)}\,dx\\
&\nn\qquad+\|v_1^{\delta(j),j}-v_1-\iQ v_1^{\delta(j),j}(x,y)\,dy\|^p_{L^p(\Omega\times Q;\R^d)}\Big).
\end{align*}
Therefore, in view of \eqref{eq:limit-set1} and \eqref{eq:limit-set2},  
 \be{eq:limit-set2p}w^j\to v_1\quad\text{strongly in }L^p(\Omega\times Q;\R^d).\ee
 We set
 $$u_{\ep}^{j}(x):=\iQ v(x,y)\,dy+w^j\Big(x,\frac{x}{\ep}\Big)\quad\text{for a.e. }x\in\Omega.$$
 By Proposition \ref{prop:simple-2-scale} and \eqref{eq:limit-set2p},
 \be{eq:limit-set-point1} u_{\ep}^{j}\sts v\quad\text{strongly two-scale in }L^p(\Omega\times Q;\R^d)\ee
 as $\ep\to 0$ and $j\to +\infty$, in this order. Moreover, by \eqref{eq:limit-set1p} and since $v\in \mathscr{F}$,
 $$\pdeor u_{\ep}^{j}=\sum_{i=1}^N A^i(x)\frac{\partial w^j}{\partial x_i}\Big(x,\frac{x}{\ep}\Big).$$
 By \eqref{eq:limit-set2p}, Proposition \ref{prop:simple-2-scale}, and the compact embedding of $W^{-1,p}$ into $L^p$, we conclude that
 \be{eq:limit-set-point2} \pdeor u_{\ep}^{j}\to \pdeor \iQ v_1(x,y)\,dy=0\ee
 strongly in $W^{-1,p}(\Omega;\R^l)$, as $\ep\to 0$ and $j\to +\infty$, in this order. By \eqref{eq:limit-set-point1}, \eqref{eq:limit-set-point2}, and Theorem \ref{thm:equivalent-two-scale} it follows in particular that 
 $$\lim_{j\to +\infty}\lim_{\ep\to 0}\Big(\|T_{\ep}u_{\ep}^{j}-v\|_{L^p(\Omega\times Q;\R^d)}+\|\pdeor u_{\ep}^{j}\|_{W^{-1,p}(\Omega;\R^l)}\Big)=0.$$
Attouch's diagonalization lemma \cite[Lemma 1.15 and Corollary 1.16]{attouch} provides us with a subsequence $\{j(\ep)\}$ such that, setting $u_{\ep}:=u_{\ep}^{j(\ep)}$, there holds
 $$T_{\ep}u_{\ep}\sts v\quad\text{strongly two-scale in }L^p(\Omega\times Q;\R^d)$$
 and
 $$\pdeor u_{\ep}\to 0\quad\text{strongly in }W^{-1,p}(\Omega;\R^l).$$
 The thesis follows applying again Theorem \ref{thm:equivalent-two-scale}.
 \end{proof}
  In order to state the main result of this section we introduce the classes
 \be{eq:def-u-f}\mathscr{U}:=\{u\in L^p(\Omega;\rd):\,\pdeor u=0\}\ee
 and
 \be{eq:def-w-f}\mathscr{W}:=\Big\{w\in L^p(\Omega\times Q;\rd): \iQ w(x,y)\,dy=0\quad\text{and }\pdeor_y w=0\Big\}.\ee
 It is clear that $v\in \mathscr{F}$ if and only if 
 $$\iQ v(x,y)\,dy\in \mathscr{U}\quad\text{and}\quad v-\iQ v(x,y)\,dy\in \mathscr{W}.$$
 Let ${\mathscr{E}}_{\rm hom}:L^p(\Omega;\rd)\to L^p(\Omega;\rd)$ be the functional
\be{eq:def-E-hom}{\mathscr{E}}_{\rm hom}(u):=
\begin{cases}\liminf_{n\to +\infty}\inf_{w\in\mathscr{W}}\iOQ f(x,ny, u(x))\,dy\,dx&\text{if }u\in\mathscr{U},\\
+\infty&\text{otherwise in }L^p(\Omega;\rd).\end{cases}\ee
 We now provide a first characterization of \eqref{eq:liminf-to-charact}.
\begin{theorem}
\label{thm:main-result-A-free}
Under the assumptions of Theorem \ref{thm:main}, for every $u\in L^p(\Omega;\R^d)$ there holds
\begin{multline}
\label{eq:main-A-free}
\inf\Big\{\liminf_{\ep\to 0}\iO f\Big(x,\frac{x}{\ep},u_{\ep}(x)\Big)\,dx:u_{\ep}\wk u\quad\text{weakly in }L^p(\Omega;\rd)\\
\text{and }\pdeor u_{\ep}\to 0\quad\text{strongly in }W^{-1,p}(\Omega;\rl)\Big\}\\
=\inf\Big\{\limsup_{\ep\to 0}\iO f\Big(x,\frac{x}{\ep},u_{\ep}(x)\Big)\,dx:u_{\ep}\wk u\quad\text{weakly in }L^p(\Omega;\rd)\\
\text{and }\pdeor u_{\ep}\to 0\quad\text{strongly in }W^{-1,p}(\Omega;\rl)\Big\}={\mathscr{E}}_{\rm hom}(u).
\end{multline}
\end{theorem}
We subdivide the proof of Theorem \ref{thm:main-result-A-free} into the proof of a limsup inequality (Corollary \ref{thm:limsup-A-free}) and a liminf inequality (Propositions \ref{thm:liminf-A-free-1} and \ref{thm:liminf-A-free-2}).

 We first show how an adaptation of the construction in Lemma \ref{lemma:A-free-fields} yields an outline for proving the limsup inequality in \eqref{eq:main-A-free}.
 \begin{proposition}
 \label{thm:limsup-basic-A-free}
 Under the assumptions of Theorem \ref{thm:main}, for every $n\in N$, $u\in\mathscr{U}$ and $w\in\mathscr{W}$ there exists a sequence $\{u_{\ep}\}\in \mathcal{S}_{u+w}$ (see \eqref{eq:def-s-v}) such that
 \ba{eq:1-thm45} 
 &u_{\ep}\wk u\quad\text{weakly in }L^p(\Omega;\rd),\\
 &\label{eq:2-thm45}\limsup_{\ep\to 0}\iO f\Big(x,\frac{x}{\ep},u_{\ep}(x)\Big)\,dx\leq \iOQ f\big(x,ny,u(x)+w(x,y)\big)\,dy\,dx.
 \end{align}
 \end{proposition}
 \begin{proof}
 
 \emph{Step 1}: We first assume that $u\in C(\Omega;\R^d)$ and $w\in C^1({\Omega}; C^1_{\rm per}(\R^N;\rd))$. \\
 Arguing as in \cite[Proof of Proposition 2.7]{fonseca.kromer} we introduce the auxiliary function
 $$g(x,y):=f(x,ny,u(x)+w(x,y))$$
 for every $x\in {\Omega}$ and for a.e. $y\in\R^N.$
 By definition, $g\in C(\Omega;L^p_{\rm per}(\R^N))$. Hence, setting
 $$g_{\ep}(x):=g\Big(x,\frac{x}{n\ep}\Big)\quad\text{for a.e. }x\in\Omega,$$
 Proposition \ref{prop:simple-2-scale} yields
 \bas
 &\lim_{\ep \to 0}\iO f\Big(x,\frac{x}{\ep}, u(x)+w\Big(x,\frac{x}{n\ep}\Big)\Big)\,dx=\lim_{\ep \to 0}\iO g_{\ep}(x)\,dx\\
 &\quad= \iOQ g(x,y)\,dy\,dx=\iOQ f(x,ny,u(x)+w(x,y))\,dy\,dx.
 \end{align*}

 Define
 $$u_{\ep}(x):=u(x)+w\Big(x,\frac{x}{n\ep}\Big)\quad\text{for a.e. }x\in{\Omega}.$$
 By the periodicity of $w$ in the second variable and by the definition of $\mathscr{W}$,
 $$u_{\ep}\wk u+\iQ w(x,y)\,dy=u\quad\text{weakly in }L^p(\Omega;\rd).$$
 By Proposition \ref{prop:simple-2-scale}, $u_{\ep}\sts u+w$ strongly two-scale in $L^p(\Omega\times Q;\R^d)$. Finally (recalling the definitions of the classes $\mathscr{U}$ and $\mathscr{W}$) by the regularity of $w$ and by Proposition \ref{prop:simple-2-scale},
$$ \pdeor u_{\ep}=\sum_{i=1}^N A^i(x)\frac{\partial w}{\partial x_i}\Big(x,\frac{x}{n\ep}\Big)\wk \sum_{i=1}^N A^i(x)\frac{\partial}{\partial_{x_i}}\iQ w(x,y)\,dy=0$$
 weakly in $L^p(\Omega;\rl)$ and hence strongly in $W^{-1,p}(\Omega;\rl)$, due to the compact embedding of $L^p$ into $W^{-1,p}$.\\
 \emph{Step 2}: Consider the general case in which $u\in \mathscr{U}$ and $w\in \mathscr{W}$. Arguing as in the second part of the proof of Lemma \ref{lemma:A-free-fields} (up to \eqref{eq:limit-set2p}), we construct a sequence $\{w^{j}\}\in C^1(\Omega; C^1_{\rm per}(\Rn;\rd))$ such that
 \be{eq:prop-strong}
 w^{j}\to u+w\quad\text{strongly in }L^p(\Omega;L^p_{\rm per}(\Rn;\rd)),
 \ee
 and
 $$\pdeor_y w^{j}= 0\quad\text{in }W^{-1,p}(\Omega;\R^l)\quad\text{for a.e. }x\in\Omega,\quad\text{for every }j.$$
 Set
 $$u_{\ep}^{j}(x):=w^{j}\Big(x,\frac{x}{n\ep}\Big)\quad\text{for a.e. }x\in\Omega.$$
 By Proposition \ref{prop:simple-2-scale}, there holds 
 \be{eq:limsup-A}u_{\ep}^{j}\sts u+w\quad\text{strongly two-scale in }L^p(\Omega\times Q;\R^d),\ee
  as $\ep\to 0$ and $j\to +\infty$, in this order. In addition, arguing as in the proof of \eqref{eq:limit-set-point2}, we have
 \be{eq:limsup-B}\pdeor u_{\ep}^{j}\to 0\quad\text{strongly in }W^{-1,p}(\Omega;\R^l)\ee
 as $\ep\to 0$ and $j\to +\infty$, in this order. 
 
 To conclude, it remains to study the asymptotic behavior of the energies associated to the sequence $\{u_{\ep}^j\}$. Consider the functions
 $$g^{j}(x,y):=f(x,ny,w^{j}(x,y))$$
 and $$g^{j}_{\ep}(x):=g^{j}\Big(x,\frac{x}{n\ep}\Big).$$
 Arguing as in Step 1, we obtain
 \ba{eq:limsup-C}
 &\lim_{j\to +\infty}\lim_{\ep\to 0}\iO f\Big(x,\frac{x}{\ep}, u_{\ep}^{j}(x)\Big)\,dx= \lim_{j\to +\infty}\lim_{\ep\to 0}\iO g_{\ep}^{j}(x)\,dx\\
 &\nn\quad=\lim_{j\to +\infty}\iOQ g^{j}(x,y)\,dy\,dx= \iOQ f(x,ny, u(x)+w(x,y))\,dy\,dx
  \end{align}
 where we used the periodicity of $g_{\ep}^{j}$, together with \eqref{eq:growth-p-f-3} and \eqref{eq:prop-strong}. In view of \eqref{eq:limsup-A}--\eqref{eq:limsup-C}, Attouch's diagonalization lemma \cite[Lemma 1.15 and Corollary 1.16]{attouch}, and Theorem \ref{thm:equivalent-two-scale}, we obtain a subsequence $\{j(\ep)\}$ such that $u_{\ep}:=u_{\ep}^{j(\ep)}$ satisfies both \eqref{eq:1-thm45} and \eqref{eq:2-thm45}. \end{proof}
 Proposition \ref{thm:limsup-basic-A-free} yields the following limsup inequality.
 \begin{corollary}
 \label{thm:limsup-A-free}
 Under the assumptions of Theorem \ref{thm:main}, for every $u\in L^p(\Omega;\R^d)$
\bas
&\inf\Big\{\limsup_{\ep\to 0}\iO f\Big(x,\frac{x}{\ep},u_{\ep}(x)\Big)\,dx:u_{\ep}\wk u\quad\text{weakly in }L^p(\Omega;\rd)\\
&\quad\text{and }\pdeor u_{\ep}\to 0\quad\text{strongly in }W^{-1,p}(\Omega;\rl)\Big\}\leq{\mathscr{E}}_{\rm hom}(u).
\end{align*}
\end{corollary}

 We now turn to the proof of the liminf inequality in Theorem \ref{thm:main-result-A-free}. For simplicity, we subdivide it into two intermediate results.
 \begin{proposition}
 \label{thm:liminf-A-free-1}
   Under the assumptions of Theorem \ref{thm:main}, for every sequence $\ep_n\to 0^+$, $u\in \mathscr{U}$ and $\{u_{n}\}\in L^p(\Omega;\rd)$ with $u_n\wk 0$ weakly in $L^p(\Omega;\rd)$ and $\pdeor u_n\to 0$, there exists a $p$-equiintegrable family of functions
 $$\mathscr{V}:=\{v_{\nu,n}:\nu,n\in\N\}$$
 such that $\mathscr{V}$ is a bounded subset of $L^p(\Omega;\R^d)$, for every $\nu\in\N$ and as $n\to +\infty$  
 \bas
 &v_{\nu,n}\wk 0\quad\text{weakly in }L^p(\Omega;\rd),\\
 &\pdeor v_{\nu,n}\to 0\quad\text{strongly in }W^{-1,q}(\Omega;\rl)\quad\text{for every }1<q<p.
 \end{align*}
 Furthermore,
 \bas
&\liminfn \iO f\Big(x,\frac{x}{\en}, u(x)+u_n(x)\Big)\,dx\geq \sup_{\nu\in\N}\Big\{\liminfn \iO f(x,\nu n x,u(x)+v_{\nu,n}(x))\,dx\Big\}.
 \end{align*}
 \end{proposition}
 \begin{proof}
 The proof follows the argument of \cite[Proof of Proposition 3.8]{fonseca.kromer}. We sketch the main steps for the convenience of the reader.\\
 \emph{Step 1}:\\
 We first truncate our sequence in order to achieve $p-$equiintegrability. Arguing as in \cite[Proof of Lemma 2.15]{fonseca.muller}, we construct a $p$-equiintegrable sequence $\{\tilde{u}_n\}\in L^p(\Omega;\rd)$ such that 
 \bas
 &\tilde{u}_n-u_n\to 0\quad\text{strongly in }L^q(\Omega;\rd)\quad\text{for every }1<q<p,\\
 &\tilde{u}_n\wk 0\quad\text{weakly in }L^p(\Omega;\rd),\\
 &\pdeor u_n\to 0\quad\text{strongly in }W^{-1,q}(\Omega;\rl)\quad\text{for every }1<q<p,
 \end{align*}
 and
 $$\liminfn \iO f\Big(x,\frac{x}{\ep_n}, u(x)+u_n(x)\Big)\,dx\geq \liminfn \iO f\Big(x,\frac{x}{\ep_n}, u(x)+\tilde{u}_n(x)\Big)\,dx.$$
 \emph{Step 2}: we consider the sequence 
 $$\tilde{k}_{\nu,n}:=\frac{1}{\nu \ep_n}.$$
  If $\{\tilde{k}_{\nu,n}\}$ is a sequence of integers (without loss of generality we can assume that it is increasing as $n$ increases), then there is nothing to prove and we simply set 
 $$v_{\nu,k}:=\begin{cases} \tilde{u}_n&\text{if }k=\tilde{k}_{\nu,n},\\
 0&\text{otherwise}.\end{cases}$$
 
 In the case in which $\{\tilde{k}_{\nu,n}\}$ is not a sequence of integers, we define
 $$k_{\nu,n}:=\frac{\theta_{\nu,n}}{\nu\ep_n},$$
 where
 $$\theta_{\nu,n}:=\nu \ep_n\floor[\Big]{\frac{1}{\nu \ep_n}}.$$
 In particular
 \be{eq:theta-nu-n}
 \theta_{\nu,n}\to 1\quad\text{as }n\to +\infty.
 \ee
 An adaptation of \cite[Lemma 2.8]{fonseca.kromer} applied to $\{\tilde{u}_n\}$ yields a $p-$equiintegrable sequence $\{\bar{u}_n\}\in L^p(Q;\rd)$ such that
 \begin{eqnarray}
 \nonumber&& \tilde{u}_n-\bar{u}_n\to 0\quad\text{strongly in }L^p(\Omega;\rd),\\
 \nonumber&&\tilde{u}_n\wk 0\quad\text{weakly in }L^p(Q\setminus \Omega;\rd),\\
 \label{eq:pdeor-bar-un}&&\pdeor \bar{u}_n\to 0\quad\text{strongly in }W^{-1,q}(Q;\rl)\quad\text{for every }1<q<p.
 \end{eqnarray}
 
 Arguing as in \cite[Proof of Proposition 3.8]{fonseca.kromer} we obtain
 $$\liminfn \iO f\Big(x,\frac{x}{\ep_n}, u+u_n(x)\Big)\,dx\geq \liminfn \iO f(x, \nu k_{\nu,n}x,u(x)+v_{\nu, k_{\nu,n}}(x))\,dx,$$
 where
 $$v_{\nu, k_{\nu,n}}(x):=\bar{u}_n(\theta_{\nu,n}x)\quad\text{for a.e. }x\in\Omega,$$
 $\nu\in \N$ and $n\in\N$ are large enough so that $\theta_{\nu,n}\Omega\subset Q$. Setting
 $$v_{\nu,n}:=\begin{cases}v_{\nu,k_{\nu,n}}&\text{if }n=k_{\nu,n},\\0&\text{otherwise},\end{cases}$$
  the sequence $\{v_{\nu, n}\}$ is uniformly bounded in $L^p(\Omega;\R^d)$, $p$-equiintegrable, and satisfies
 $$v_{\nu, n}\wk 0\quad\text{weakly in }L^p(\Omega;\rd)$$
 as $n\to +\infty$. To conclude, it remains only to show that
  \be{eq:thesis-pde-liminf}\pdeor v_{\nu,n}\to 0\quad\text{strongly in }W^{-1,q}(\Omega;\rl)\quad\text{for every }1<q<p
  \ee
  as $n\to +\infty$.
  
  Let $q$ as above be fixed, and let $\varphi\in W^{1,q'}_0(\Omega;\rl)$. A change of variables yields
  \bas
  &|\scal{\pdeor v_{\nu,k_{\nu,n}}}{\varphi}|\\
  &\quad=\Big|\iO \Big(\sum_{i=1}^N A^i(x)\bar{u}_n(\theta_{\nu,n}x)\cdot\frac{\partial \varphi(x)}{\partial x_i}+\sum_{i=1}^N\frac{\partial A^i(x)}{\partial x_i}\bar{u}_n(\theta_{\nu,n}x)\cdot\varphi(x)\Big)\,dx\Big|\\
 &\quad=\frac{1}{\theta_{\nu,N}^{N}}\Big|\int_{\theta_{\nu,N}\Omega} \Big(\sum_{i=1}^N A^i\Big(\frac{y}{\theta_{\nu,n}}\Big)\bar{u}_n(y)\cdot\frac{\partial\varphi}{\partial x_i}\Big(\frac{y}{\theta_{\nu,n}}\Big)\\
 &\qquad+\sum_{i=1}^N\frac{\partial A^i}{\partial x_i}\Big(\frac{y}{\theta_{\nu,n}}\Big)\bar{u}_n(y)\cdot\varphi\Big(\frac{y}{\theta_{\nu,n}}\Big)\Big)\,dy\Big|.
\end{align*}
For $n$ big enough $\theta_{\nu,N}\Omega\subset Q$. Hence, by \eqref{eq:theta-nu-n}, adding and subtracting the quantity
 $$\frac{1}{\theta_{\nu,N}^{N}}\int_{\theta_{\nu,N}\Omega} \Big(\sum_{i=1}^N A^i(y)\bar{u}_n(y)\cdot\frac{\partial\varphi}{\partial x_i}\Big(\frac{y}{\theta_{\nu,n}}\Big)+\sum_{i=1}^N\frac{\partial A^i}{\partial x_i}(y)\bar{u}_n(y)\cdot\varphi\Big(\frac{y}{\theta_{\nu,n}}\Big)\Big)\,dy,$$
we deduce the upper bound
 \bas
 & |\scal{\pdeor v_{\nu,n}}{\varphi}|\\
 &\quad \leq C\sum_{i=1}^N\Big\|A^i(y)-A^i\Big(\frac{y}{\theta_{\nu,N}}\Big)\Big\|_{C^0(\bar{Q};\M^{l\times d})}\|\bar{u}_n\|_{L^q(Q;\rd)}\|\varphi\|_{W^{1,q'}_0(\Omega;\rl)}\\
 &\qquad +C\sum_{i=1}^N\Big\|\frac{\partial A^i}{\partial y_i}(y)-\frac{\partial A^i}{\partial x_i}\Big(\frac{y}{\theta_{\nu,N}}\Big)\Big\|_{C^0(\bar{Q};\M^{l\times d})}\|\bar{u}_n\|_{L^q(Q;\rd)}\|\varphi\|_{W^{1,q'}_0(\Omega;\rl)}\\
 &\qquad +C\|\pdeor \bar{u}_n\|_{W^{-1,q}(\Omega;\rl)}\|\varphi\|_{W^{1,q'}_0(\Omega;\rl)}.
 \end{align*}
 Property \eqref{eq:thesis-pde-liminf} follows now by \eqref{eq:theta-nu-n} and \eqref{eq:pdeor-bar-un}.
 \end{proof}
 To complete the proof of the liminf inequality in \eqref{eq:main-A-free} we apply the \emph{unfolding operator} (see Subsection \ref{subsection:unfolding}) to the set $\mathscr{V}$ constructed in Proposition \ref{thm:liminf-A-free-1}. 
 \begin{proposition}
 \label{thm:liminf-A-free-2}
 Under the assumptions of Theorem \ref{thm:main}, for every $u\in\mathscr{U}$ and every family $\mathscr{V}=\{v_{\nu,n}:\,\nu, n\in \N\}$ as in Proposition \ref{thm:liminf-A-free-1}  there holds
 $$\liminf_{\nu\to +\infty}\liminfn \iO f(x,\nu n x, u(x)+v_{\nu,n}(x))\,dx\geq {\mathscr{E}}_{\rm hom}(u).$$
 \end{proposition}
 \begin{proof}
 Fix $u\in \mathscr{U}$ and let $\{v_{\nu,n}:\,\nu, n\in \N\}$ be $p-$equiintegrable and bounded in $L^p(\Omega;\R^d)$, with 
 \be{eq:vnun-wk}
 v_{\nu,n}\wk 0\quad\text{weakly in }L^p(\Omega;\rd)
 \ee
 and 
\be{eq:vnun-pde}
 \pdeor v_{\nu,n}\to 0\quad\text{strongly in }W^{-1,q}(\Omega;\rl)\quad\text{for every }1<q<p,
 \ee as $n\to +\infty$, for every $\nu\in\N$. Fix $\Omega'\subset\subset \Omega$ and for $z\in\mathbb{Z}^N$ and $n\in\N$, define
 $$Q_{\nu,z}:=\frac{z}{\nu}+\frac{1}{\nu}Q,$$
 and
 $$Z^{\nu}:=\{z\in \mathbb{Z}^N: Q_{\nu,z}\cap \Omega'\neq\emptyset\}.$$
We consider the maps
  $$T_{\frac{1}{\nu}}v_{\nu,n}(x,y):=v_{\nu,n}\Big(\frac{1}{\nu}\floor{\nu x}+\frac{1}{\nu}y\Big)\quad\text{for a.e. }x\in\Omega, y\in Q,$$
where we have extended the sequence $\{v_{\nu,n}\}$ to zero outside $\Omega$. A change of variables yields
\bas
&\int_{\Omega}f(x,\nu n x,u(x)+v_{\nu,n}(x))\,dx\geq \sum_{z\in\Z^{\nu}}\int_{Q_{\nu,z}}f(x,\nu n x,u(x)+v_{\nu,n}(x))\,dx\\
&\quad=\nu^N \sum_{z\in\Z^{\nu}}\iQ f\Big(\frac{z}{\nu}+\frac{y}{\nu}, ny, u\Big(\frac{z}{\nu}+\frac{y}{\nu}\Big)+v_{\nu,n}\Big(\frac{z}{\nu}+\frac{y}{\nu}\Big)\Big)\,dy\\
&\quad=\sum_{z\in\Z^{\nu}}\int_{Q_{\nu,z}}\iQ f\Big(\frac{\floor{\nu x}}{\nu}+\frac{y}{\nu}, ny, T_{\frac{1}{\nu}}u(x,y)+T_{\frac{1}{\nu}}v_{\nu,n}(x,y)\Big)\,dy\,dx\\
&\quad\geq \sum_{z\in\Z^{\nu}}\int_{Q_{\nu,z}\cap\Omega'}\iQ f\Big(\frac{\floor{\nu x}}{\nu}+\frac{y}{\nu}, ny, T_{\frac{1}{\nu}}u(x,y)+T_{\frac{1}{\nu}}v_{\nu,n}(x,y)\Big)\,dy\,dx,
\end{align*}
where the last inequality is due to \eqref{eq:growth-p-f-3}. By \cite[Proposition 3.6 (i)]{fonseca.kromer} and Proposition \ref{prop:conv-unf-op} we conclude that
\ba{eq:almost-final-liminf}
 &\iO f(x,\nu n x, u(x)+v_{\nu,n}(x))\,dx\\
 \nn&\quad\geq \sigma_{\nu}+\sum_{z\in Z^{\nu}}\int_{Q_{\nu,z}\cap \Omega'}\iQ f(x,ny,u(x)+\hat{v}_{\nu,z,n}(y))\,dy\,dx,
 \end{align}
 where $$\hat{v}_{\nu,z,n}(y):=T_{\frac{1}{\nu}}v_{\nu,n}\Big(\frac{z}{\nu},y\Big)$$
 for a.e. $y\in Q$, and $\sigma_{\nu}\to 0$ as $\nu\to +\infty$. The sequence $\{\hat{v}_{\nu,z,n}\}$ is $p$-equiintegrable by \cite[Proposition A.2]{fonseca.kromer}, and is uniformly bounded by \eqref{eq:vnun-wk} and Proposition \ref{prop:isometry}, since 
 $$\iQ|\hat{v}_{\nu,z,n}(y)|^p\,dy=\frac{1}{\nu^N}\int_{Q_{\nu,z}}|v_{\nu,n}(x)|^p\,dx.$$
 By the boundedness of $\{v_{\nu,n}:\,\nu,n\in\N\}$ in $L^p(\Omega;\R^d)$,
  and by \eqref{eq:vnun-wk} there holds
 \be{eq:vznun-wk}\hat{v}_{\nu,z,n}\wk 0\quad\text{weakly in }L^p(Q;\rd)\ee
 as $n\to +\infty$, for every $z\in Z^{\nu}$, $\nu\in\N$. Denoting by $\chi_{Q_{\nu,z}\cap\Omega'}$ the characteristic functions of the sets ${Q_{\nu,z}\cap\Omega'}$, we claim that
 \be{eq:most-difficult-estimate}
 \limsup_{\nu\to +\infty}\limsup_{n\to+\infty}\Big\|\Big\|\pdeor_y(x)\sum_{\nu\in\Z^{\nu}}\chi_{Q_{\nu,z}\cap\Omega'}(x)\hat{v}_{\nu,z,n}(y)\Big\|_{W^{-1,q}(Q;\rl)}\Big\|_{L^q(\Omega)}=0\ee
 for every $1< q<p$. Indeed, fix $1<q<p$, and let $\psi\in W^{1,q'}_0(Q;\rl)$. Then
 \bas
\Big|\scal{\pdeor_y\Big(\frac{z}{\nu}\Big)\hat{v}_{\nu,z,n}}{\psi}\Big|&= \Big|\iQ\sum_{i=1}^N A^i\Big(\frac{z}{\nu}\Big)v_{\nu,n}\Big(\frac{z}{\nu}+\frac{y}{\nu}\Big)\cdot\frac{\partial \psi(y)}{\partial y_i}\,dy\Big|\\
 &=\nu^N\Big|\int_{Q_{\nu,z}}\sum_{i=1}^N A^i\Big(\frac{z}{\nu}\Big)v_{\nu,n}(x)\cdot\frac{\partial\psi}{\partial y_i}(\nu x-z)\,dx\Big|.
 \end{align*}
 Adding and subtracting to the previous expression the quantity
 $$\nu^N\int_{Q_{\nu,z}}\sum_{i=1}^N A^i(x)v_{\nu,n}(x)\cdot\frac{\partial\psi}{\partial y_i}(\nu x-z)\,dx,$$
 and setting $\phi^{\nu}_z(x):=\psi(\nu x-z)$ for a.e. $x\in\Omega$,
 we obtain the estimate
 \begin{align*}
& \Big|\scal{\pdeor_y\Big(\frac{z}{\nu}\Big)\hat{v}_{\nu,z,n}}{\psi}\Big|\\
&\quad\leq \nu^N\Big\|\sum_{i=1}^N \Big(A^i\Big(\frac{z}{\nu}\Big)-A^i(x)\Big)v_{\nu,n}(x)\Big\|_{L^q(Q_{\nu,z};\rl)}\Big\|\frac{\partial\psi}{\partial y_i}(\nu x-z)\Big\|_{L^{q'}(Q_{\nu,z};\rl)}\\
&\qquad +\nu^{N-1}\Big\|\frac{\partial}{\partial x_i}(A^i(x)v_{\nu,n}(x))\Big\|_{W^{-1,q}(Q_{\nu,z};\rl)}\|\phi^{\nu}_z\|_{W^{1,q'}_0(Q_{\nu,z};\rl)}.
 \end{align*}
 A change of variables yields the upper bound
 $$\Big\|\frac{\partial\psi}{\partial y_i}(\nu x-z)\Big\|_{L^{q'}(Q_{\nu,z};\rl)}+\frac{\|\phi^\nu_z\|_{W^{1,q'}_0(Q_{\nu,z};\rl)}}{\nu}
 \leq \frac{C}{\nu^{\frac{N}{q'}}}\|\psi\|_{W^{1,q'}_0(Q;\rl)}.$$
 Thus, by the regularity of the operators $A^i$,
 \ba{eq:compl1}
 \Big|\scal{\pdeor_y\Big(\frac{z}{\nu}\Big)\hat{v}_{\nu,z,n}}{\psi}\Big|&\leq C\nu^{\frac{N}{q}-1}\Big\|\sum_{i=1}^N\frac{\partial A^i}{\partial x_i}\Big\|_{L^{\infty}(Q;\M^{l\times d})}\|v_{\nu,n}\|_{L^q(Q_{\nu,z};\rd)}\|\psi\|_{W^{1,q'}_0(Q;\rl)}\\
 \nn&\quad+C\nu^{\frac{N}{q}}\Big\|\frac{\partial}{\partial x_i}(A^i(x)v_{\nu,n}(x))\Big\|_{W^{-1,q}(Q_{\nu,z};\rl)}\|\psi\|_{W^{1,q'}_0(Q;\rl)}.
 \end{align}
  Using again the Lipschitz regularity of the operators $A^i$, $i=1,\cdots,N$, we deduce
 \ba{eq:compl2}
& \|\pdeor_y(x)\hat{v}_{\nu,z,n}(y)\|_{W^{-1,q}(Q;\rl)}\leq \sum_{i=1}^N \Big\|A^i(x)-A^i\Big(\frac{z}{\nu}\Big)\Big\|_{L^{\infty}(Q;\M^{l\times d})}\|\hat{v}_{\nu,z,n}\|_{L^q(Q;\rl)}\\
 \nn&\qquad+\Big\|\pdeor_y\Big(\frac{z}{\nu}\Big)\hat{v}_{\nu,z,n}(y)\Big\|_{W^{-1,q}(Q;\rl)}\\
\nn&\quad \leq \frac{C}{\nu}\Big\|\sum_{i=1}^N\frac{\partial A^i}{\partial x_i}\Big\|_{L^{\infty}(Q;\M^{l\times d})}\|\hat{v}_{\nu,z,n}\|_{L^q(Q;\rl)}+\Big\|\pdeor_y\Big(\frac{z}{\nu}\Big)\hat{v}_{\nu,z,n}(y)\Big\|_{W^{-1,q}(Q;\rl)}
 \end{align}
 for a.e. $x\in Q_{\nu,z}$. Hence, by \eqref{eq:compl1} and \eqref{eq:compl2}, we obtain
 \bas
 &\sum_{z\in\Z^{\nu}}\int_{Q_{\nu,z}\cap\Omega'}\|\pdeor_y(x)\hat{v}_{\nu,z,n}(y)\|_{W^{-1,q}(Q;\R^l)}^q\,dx\\
&\quad \leq \sum_{z\in\Z^{\nu}}\frac{C}{\nu^q}\int_{Q_{\nu,z}\cap\Omega'} \Big\|\sum_{i=1}^N\frac{\partial A^i}{\partial x_i}\Big\|^q_{L^{\infty}(Q;\M^{l\times d})}\|\hat{v}_{\nu,z,n}\|^q_{L^q(Q;\R^d)}\,dx\\
&\qquad+\sum_{z\in\Z^{\nu}}\int_{Q_{\nu,z}\cap\Omega'}\|\pdeor_y\Big(\frac{z}{\nu}\Big)\hat{v}_{\nu,z,n}(y)\|_{W^{-1,q}(Q;\R^l)}^q\,dx\\
&\quad\leq \frac{C}{\nu^q}\Big\|\sum_{z\in \Z^{\nu}}\chi_{Q_{\nu,z}\cap\Omega'}(x)\hat{v}_{\nu,z,n}(y)\Big\|_{L^q(\Omega\times Q;\R^d)}^q+\frac{C}{\nu^q}\|v_{\nu,n}\|_{L^q(\Omega;\R^d)}\\
&\qquad+\frac{C}{\nu^q}\sum_{z\in \Z^{\nu}}\int_{Q_{\nu,z}\cap\Omega'}(x)\hat{v}_{\nu,z,n}\Big\|\sum_{i=1}^N\frac{\partial A^i v_{\nu,n}}{\partial x_i}\Big\|_{W^{-1,q}(Q_{\nu,z};\rl)}\\
&\quad\leq \frac{C}{\nu^q}\Big\|\sum_{z\in \Z^{\nu}}\chi_{Q_{\nu,z}\cap\Omega'}(x)\hat{v}_{\nu,z,n}(y)\Big\|_{L^q(\Omega\times Q;\R^d)}^q+\frac{C}{\nu^q}\|v_{\nu,n}\|_{L^q(\Omega;\R^d)}\\
&\qquad+C\nu^{\frac{N}{q}}\Big\|\sum_{i=1}^N\frac{\partial A^i v_{\nu,n}}{\partial x_i}\Big\|_{W^{-1,q}(\Omega;\rl)}
 \end{align*}
 Property \eqref{eq:most-difficult-estimate} follows now by \eqref{eq:vnun-wk} and \eqref{eq:vnun-pde}, and by the compact embedding of $L^p$ into $W^{-1,p}$.

 Consider the maps
 $$w_{\nu,n}(x,y):=\begin{cases}\Pi(x)\Big(\hat{v}_{\nu,z,n}(y)-\iQ \hat{v}_{\nu,z,n}(\xi)\,d\xi\Big)-\iQ\Pi(x)\Big(\hat{v}_{\nu,z,n}(y)-\iQ \hat{v}_{\nu,z,n}(\xi)\,d\xi\Big)\,dy\\
 \qquad\qquad\text{for }x\in Q_{\nu,z}\cap\Omega',\,z\in Z^{\nu},\,y\in Q,\\
 0\quad\qquad\phantom{ii}\text{otherwise in }\Omega.\end{cases}$$
 By Lemma \ref{lemma:proj-operator} the sequence $\{w_{\nu,n}\}$ is $p$-equiintegrable, and
 $$\pdeor_y w_{\nu,n}=0\quad\text{in }W^{-1,p}(Q;\R^l)\quad\text{for a.e. }x\in\Omega,$$
 for all $\nu,n\in N$. In particular, $\{w_{\nu,n}\}\subset \mathscr{W}$. We claim that
 \be{eq:replace-vznun}
 \Big\|w_{\nu,n}(x,y)-\sum_{z\in Z^{\nu}}\chi_{Q_{\nu,z}\cap\Omega'}(x)\hat{v}_{\nu,z,n}(y)\Big\|_{L^q(\Omega\times Q;\rd)}\to 0
 \ee
 as $n\to+\infty$, $\nu\to +\infty$, for every $1<q<p$.
 
 In fact, by Lemma \ref{lemma:proj-operator} there holds
 \bas
 \Big\|\Pi(x)\Big(\hat{v}_{\nu,z,n}(y)-\iQ\hat{v}_{\nu,z,n}(\xi)\,d\xi\Big)-\hat{v}_{\nu,z,n}(y)\Big\|^q_{L^q(Q;\R^d)}\\
 \leq C\Big(\|\pdeor_y(x)\hat{v}_{\nu,z,n}(y)\|^q_{W^{-1,q}(Q;\R^l)}+\Big|\iQ\hat{v}_{\nu,z,n}(y)\,dy\Big|^q\Big).
 \end{align*}
 Therefore
 \ba{eq:liminf-star}
 &\Big\|\sum_{z\in\Z^{\nu}}\chi_{Q_{\nu,z}\cap\Omega'}(x)\Big(\Pi(x)\Big(\hat{v}_{\nu,z,n}(y)-\iQ\hat{v}_{\nu,z,n}(\xi)\,d\xi\Big)-\hat{v}_{\nu,z,n}(y)\Big)\Big\|^q_{L^q(\Omega\times Q;\R^d)}\\
 \nn&\leq C\Big(\sum_{z\in\Z^{\nu}}\int_{Q_{\nu,z}\cap\Omega'}\|\pdeor_y(x)\hat{v}_{\nu,z,n}(y)\|^q_{W^{-1,q}(Q;\R^l)}\,dx\\
 \nn&\quad+\sum_{z\in\Z^{\nu}}\int_{Q_{\nu,z}\cap\Omega'}\Big|\iQ\hat{v}_{\nu,z,n}(y)\,dy\Big|^q\,dx\Big).
 \end{align}
 The first term in the right-hand side of \eqref{eq:liminf-star} converges to zero as $n\to +\infty$ and $\nu\to +\infty$, in this order, owing to \eqref{eq:most-difficult-estimate}.
 The second term in the right-hand side of \eqref{eq:liminf-star} converges to zero as $n\to +\infty$ and $\nu\to +\infty$, in this order, by the dominated convergence theorem, owing to \eqref{eq:vznun-wk} and the uniform boundedness in $L^p$ of $\{\hat{v}_{\nu,z,n}\}$. Hence, both the left-hand side of \eqref{eq:liminf-star} and the quantity
 $$\iQ\Big\{\sum_{z\in\Z^{\nu}}\chi_{Q_{\nu,z}\cap\Omega'}(x)\Pi(x)\Big(\hat{v}_{\nu,z,n}(y)-\iQ\hat{v}_{\nu,z,n}(\xi)\,d\xi\Big)\Big\}\,dy\to 0,$$
 converge to zero as $n\to +\infty$ and $\nu\to +\infty$, and we obtain \eqref{eq:replace-vznun}.
 
 Up to the extraction of a (not relabeled) subsequence, we can assume that
 \ba{eq:en-is-lim}
& \liminf_{\nu\to +\infty}\liminfn \sum_{z\in Z^{\nu}}\int_{Q_{\nu,z}\cap \Omega'}\iQ f(x,ny,u(x)+\hat{v}_{\nu,z,n}(y))\,dy\,dx\\
 \nn&\quad =\lim_{\nu\to +\infty}\liminfn \sum_{z\in Z^{\nu}}\int_{Q_{\nu,z}\cap \Omega'}\iQ f(x,ny,u(x)+\hat{v}_{\nu,z,n}(y))\,dy\,dx.
  \end{align}
Hence, in view of \eqref{eq:replace-vznun} and \eqref{eq:en-is-lim} we can extract a subsequence $\{n(\nu)\}$ such that
\ba{eq:point-en}
\lim_{\nu\to +\infty}\liminf_{n\to +\infty} \sum_{z\in Z^{\nu}}\int_{Q_{\nu,z}\cap \Omega'}\iQ f(x,ny,u(x)+\hat{v}_{\nu,z,n}(y))\,dy\,dx\\
=\lim_{\nu\to +\infty}\sum_{z\in Z^{\nu}}\int_{Q_{\nu,z}\cap \Omega'}\iQ f(x,n(\nu)y,u(x)+\hat{v}_{\nu,z,n(\nu)}(y))\,dy\,dx,
\end{align}
and
\be{eq:point2-en}
w_{\nu,n(\nu)}(x,y)-\sum_{z\in Z^{\nu}}\chi_{Q_{\nu,z}\cap\Omega'}(x)\hat{v}_{\nu,z,n(\nu)}(y)\to 0\quad\text{strongly in }L^q(\Omega\times Q;\rd),
\ee
for every $1<q<p$.
  Going back to \eqref{eq:almost-final-liminf}, by \cite[Proposition 3.5 (ii)]{fonseca.kromer}, \eqref{eq:point-en} and \eqref{eq:point2-en},
 \bas
 &\liminf_{\nu\to +\infty}\liminf_{n\to+\infty}\iO f(x,\nu n x, u(x)+v_{\nu,n}(x))\,dx\\
 &\quad\geq \liminf_{\nu\to +\infty}\int_{\Omega'}\iQ f(x, n(\nu) y, u(x)+w_{\nu,n(\nu)}(x,y))\,dy\,dx.
 \end{align*}
 By the $p$-equiintegrability of $\{w_{\nu,n(\nu)}\}$ and by \eqref{eq:growth-p-f-3}, letting $|\Omega\setminus\Omega'|$ tend to zero, we conclude
 \bas
 &\liminf_{\nu\to +\infty}\liminf_{n\to+\infty}\iO f(x,\nu n x, u(x)+v_{\nu,n}(x))\,dx\\
  &\quad\geq \liminf_{\nu\to +\infty}\int_{\Omega}\iQ f(x, n(\nu) y, u(x)+w_{\nu,n(\nu)}(x,y))\,dy\,dx\\
 &\quad\geq \liminf_{\nu\to +\infty} \inf_{w\in \mathscr{W}}\iOQ f(x,n(\nu)y,w(x,y))\,dy\,dx\\
 &\quad\geq \liminf_{n\to +\infty}\inf_{w\in \mathscr{W}}\iOQ f(x,ny,w(x,y))\,dy\,dx={\mathscr{E}}_{\rm hom}(u).
 \end{align*}
 \end{proof}
  \begin{proof}[Proof of Theorem \ref{thm:main-result-A-free}]
  The proof follows by combining Corollary \ref{thm:limsup-A-free} with Propositions \ref{thm:liminf-A-free-1} and \ref{thm:liminf-A-free-2}.
 \end{proof}
 \begin{corollary}
 \label{cor:local}
Under the same assumptions of Theorem \ref{thm:main-result-A-free}, for every $u\in\mathscr{U}_F$
$${\mathscr{E}}_{\rm hom}(u)=\iO f_{\rm hom}(x,u(x))\,dx,$$
where
$$f_{\rm hom}(x,u(x))=\liminfn \inf_{v\in\C_x}\iQ f(x,ny,u(x)+v(y))\,dy,$$
and $\C_x$ is the class defined in \eqref{eq:def-c-x}.
\end{corollary}
\begin{proof}
We omit the proof of this corollary as it follows from \cite[Remark 3.3 (ii)]{fonseca.kromer} and by adapting the arguments in \cite[Corollary 3.2]{fonseca.kromer} and Lemma \ref{lemma:measurability} below.
\end{proof}
 \begin{proof}[Proof of Theorem \ref{thm:main}]
 The thesis results from Theorem \ref{thm:main-result-A-free} and Corollary \ref{cor:local}.
 \end{proof}
We conclude this section by showing that Theorem \ref{thm:main-result-A-free} yields a relaxation result in the framework of $\pdeor-$quasiconvexity with variable coefficients. Before stating the corollary, we prove a preliminary lemma which guarantees the measurability of the function $x\mapsto \qa f(x,u(x))$ for every $u\in L^p(\Omega;\rd)$.
 \begin{lemma}
\label{lemma:measurability}
Let $1<p<+\infty$, $u\in L^p(\Omega;\rd)$, let $\pdeor$ be as in Theorem \ref{thm:main}, and let $f:\Omega\times \rd\to [0,+\infty)$ be a Carath\'eodory function satisfying 
$$0\leq f(x,\xi)\leq C(1+|\xi|^p)\quad\text{for a.e. }x\in\Omega\times\rd,\text{ and for all }\xi\in\R^d.$$
 Then the map
$$x\mapsto \qa f(x,u(x))$$
is measurable in $\Omega$.
\end{lemma}
\begin{proof}
We first remark that
\be{eq:bounded-test-fc}
\qa f(x,u(x))=\inf_{r\in (0,+\infty)} \qa^r f(x,u(x))\quad\text{for a.e. }x\in\Omega,
\ee
where
\bas
\qa^r f(x,u(x)):=\inf &\Big\{ \iQ f(x,u(x)+w(y))\,dy:\, w\in\C_x\text{ and }\|w\|_{L^p(Q;\rd)}\leq r\Big\},
\end{align*}
and $\cx$ is the class defined in \eqref{eq:def-c-x}.
Clearly $$\qa^{r} f(x,u(x))\geq \qa f(x,u(x))$$
for a.e. $x\in\Omega$, for every $r\in\N$. Moreover, for every $\ep>0$ there exists $w_{\ep}\in \C_x$ such that
\bas
\qa f(x,u(x))&\geq \iQ f(x,u(x)+w_{\ep}(y))\,dy-\ep\\
&\quad\geq \qa^{\|w_{\ep}\|_{L^p(Q;\rd)}}f(x,u(x))-\ep\geq \inf_{r\in\N}\qa^r f(x,u(x))-\ep,
\end{align*}
which in turn implies the second inequality in \eqref{eq:bounded-test-fc}. 

By \eqref{eq:bounded-test-fc} it is enough to show that $x\mapsto \qa^r f(x,u(x))$ is measurable for every $r\in\N$. We claim that
\be{eq:same-set}
\qa^r f(x,u(x))=\sup_{n\in\N} \qa^{r,n} f(x,u(x))\quad\text{for a.e. }x\in\Omega,
\ee
where
\bas
\qa^{r,n} f(x,u(x))=\inf&\Big\{\iQ f(x,u(x)+w(y))\,dy+n\|\pdeor_y(x) w(x,\cdot)\|_{W^{-1,p}(Q;\rl)}:\\
&\quad w\in L^p(Q;\rd),\,\iQ w(y)\,dy=0\text{ and }\|w\|_{L^p(Q;\rd)}\leq r\Big\}.
\end{align*}
Clearly,
$$\qa^{r,n}f(x,u(x))\leq \qa^r f(x,u(x))$$
for a.e. $x\in\Omega$, for all $n\in\N$. To prove the opposite inequality, fix $x\in\Omega$, and for every $n\in\N$, let $w_n\in L^p(Q;\rd)$, with $\iQ w_n(y)\,dy=0$ and $\|w_n\|_{L^p(Q;\rd)}\leq r$, be such that
\ba{eq:wn-pde}
&\iQ f(x,u(x)+w_n(y))\,dy+n\|\pdeor_y(x) w_n\|_{W^{-1,p}(Q;\rl)}\\
&\nonumber\quad\leq \qa^{r,n} f(x,u(x))+\frac{1}{n}\leq f(x,u(x))+\frac{1}{n}
\end{align}
(the last inequality holds because $0\in\C_x$ for every $x\in\Omega$). Since $\{w_n\}$ is uniformly bounded in $L^p(Q;\rd)$ and by \eqref{eq:wn-pde}
$$\|\pdeor_y(x) w_n\|_{W^{-1,p}(Q;\rl)}\to 0$$
as $n\to +\infty$, there exists a map $w\in L^p(Q;\rd)$, with $\iQ w(y)\,dy=0$, $\|w\|_{L^p(Q;\rd)}\leq r$ and $\pdeor_y(x) w=0$ such that
$$w_n\wk w\quad\text{weakly in }L^p(Q;\rd).$$
By \cite[Lemma 3.1]{braides.fonseca.leoni} we can construct a sequence $\{\tilde{w}_n\}$ such that $\iQ \tilde{w}_n(y)\,dy=0$, $\pdeor_y(x) \tilde{w}_n=0$ for every $n\in\N$, and
\bas
\liminfn \iQ f(x,u(x)+\tilde{w}_n(y))\,dy&\leq \liminfn \iQ f(x,u(x)+w_n(y))\,dy\\
&\leq\sup_{n\in\N}\qa^{r,n} f(x,u(x)).
\end{align*}
In view of \eqref{eq:bounded-test-fc} we have
$$\qa^r f(x,u(x))\leq \qa f(x,u(x))\leq \iQ f(x,u(x)+\tilde{w}_n(y))\,dy\quad\text{for every }n\in\N,$$
and we obtain the second inequality in \eqref{eq:same-set}. By the measurability of 
$$x\mapsto \qa^{r,n}f(x,u(x))$$
for every $r,n\in\N$ (we can reduce it to a countable pointwise infimum of measurable functions), we deduce the measurability of
$$x\mapsto \qa^r f(x,u(x))$$
for every $r\in\N$, which in turn implies the thesis.
\end{proof}
For every $D\in\mathcal{O}(\Omega)$ and $u\in L^p(\Omega;\R^d)$, define
\begin{align}
\label{eq:def-I}
\mathcal{I}(u,D):=\inf\Big\{&\liminf_{n\to +\infty}\int_{D}f(x,u_n(x)):\,u_n\wk u\quad\text{weakly in }L^p(\Omega;\R^m)\, \\
&\nonumber\text{and }\pdeor u_n\to 0\quad\text{strongly in }W^{-1,p}(\Omega;\R^l)\Big\}.
\end{align}
Corollary \ref{cor:local} provides us with the following integral representation of $\mathcal{I}$.
\begin{corollary}
\label{thm:relax}
Let $1<p<+\infty$ and let $\pdeor$ be as in Theorem \ref{thm:main}.
Let $f:\Omega\times \rd\to [0,+\infty)$ be a Carath\'eodory function satisfying 
$$0\leq f(x,\xi)\leq C(1+|\xi|^p)\quad\text{for a.e. }x\in\Omega,\text{ and for all }\xi\in\R^d.$$
Then
$$\int_D \qa f(x,u(x))\,dx=\mathcal{I}(u,D)$$
for all $D\in\Oo(\Omega)$ and $u\in L^p(\Omega;\rd)$ with $\pdeor u=0$.
\end{corollary}

\section*{acknowledgements}
The authors thank the Center for Nonlinear Analysis (NSF Grant No. DMS-0635983), where this research was carried out, and
also acknowledge support of the National Science Foundation under the PIRE Grant No.
OISE-0967140. The research of
I. Fonseca and E. Davoli was funded by the National Science Foundation under Grant No. DMS-
0905778. The research of E. Davoli was also supported by the Austrian FWF project ``Global variational methods for nonlinear evolution equations''. E. Davoli is a member of the INdAM-GNAMPA Project 2015 ``Critical phenomena in the mechanics of materials: a variational approach''. The research of I. Fonseca was further partially supported by the National Science Foundation under Grant No. DMS-1411646.
\bibliographystyle{plain}
\bibliography{ed}
\end{document}